\documentclass[12pt,reqno]{amsart}
\usepackage[english]{babel}
\usepackage{eucal}
\usepackage{mathrsfs,amsmath,amssymb,latexsym,amsthm,amsfonts, fullpage}
\usepackage[usenames, dvipsnames]{color}
\usepackage{hyperref}
\usepackage{pdfsync}

\usepackage[utf8]{inputenc}
\usepackage[normalem]{ulem}

\title{A new  microlocal analysis of hyperfunctions}
\author {Gustavo Hoepfner}
\address{Departamento de Matem\'atica, Universidade Federal de S\~ao Carlos, S\~ao Carlos, SP, 13565-905, Brasil}
\email{hoepfner@dm.ufscar.br}

\author {Luis F. Ragognette}
\address{Departamento de Matem\'atica, Universidade Federal de S\~ao Carlos, S\~ao Carlos, SP, 13565-905, Brasil}
\email{luisragognette@dm.ufscar.br}

\date{}
 
\thanks{Work supported in part by  CNPq (grant number 305746/2015-4) and FAPESP (grant numbers 2016/13620-5, 2017/03825-1, 2017/06993-2 and 2017/13450-5).}


\theoremstyle{definition} 

\newtheorem{Def}{Definition}[section]

\theoremstyle{definition}
\newtheorem{Exe}[Def]{Example}

\newtheorem{Rem}[Def]{Remark}

\theoremstyle{plain}
\newtheorem{Pro}[Def]{Proposition}
\newtheorem{Cor}[Def]{Corollary}
\newtheorem{Teo}[Def]{Theorem}
\newtheorem{Lem}[Def]{Lemma}

\newcommand{\bb}{\mathrm{b}}

\newcommand{\RN}{\R^N}

\newcommand{\CN}{\C^N}
\newcommand{\ZN}{\Z^N}
\newcommand{\dist}{\mathrm{dist}}

\newcommand{\R}{\mathbb{R}}
\newcommand{\N}{\mathbb{N}}

\newcommand{\Rn}{\R^{n}}
\newcommand{\C}{\mathbb{C}}

\newcommand{\Z}{\mathbb{Z}}

\newcommand{\Cinf}{C^{\infty}}

\newcommand{\lra}{\longrightarrow}

\newcommand{\del}{\partial}

\newcommand{\z}{\bar{z}}
\renewcommand{\Re}{\mathrm{Re}\,}
\renewcommand{\Im}{\mathrm{Im}\,}
\newcommand{\dd}{\textnormal{d}}

\newcommand{\Char}{\mathrm{Char}}

\numberwithin{equation}{section}

\begin{document}

\begin{abstract}
  In this work we study microlocal regularity of hyperfunctions defining in this context a  class of generalized FBI transforms first introduced for distributions by Berhanu and Hounie. Using a microlocal decomposition of a hyperfunction and the generalized FBI transforms we were able to characterize the wave-front set of hyperfunctions according several types of regularity. The microlocal decomposition allowed us to recover and generalize both classical and recent results and, in particular, we proved  for differential operators with real-analytic coefficients that if the elliptic regularity theorem  regarding any reasonable regularity holds for distributions, then it is automatically true for hyperfunctions. 
\end{abstract}

\maketitle

\setcounter{tocdepth}{1}

\tableofcontents

\section{Introduction}

The goal of this paper is to study microlocal regularity of hyperfunctions introducing to this context a  class of generalized FBI transforms that first appeared for distributions in a work of Berhanu and Hounie~\cite{BH}. 
Important applications involving qualitative theory of PDE's with the use of the powerful machinery of hyperfunctions has been developed in the last 20 years \cite{CoTrep98, CoHan09, CoHan12} and the references therein. In the last year it was also the subject of two important works. Indeed, in \cite{AraCo19}, the authors use hyperfunctions theory to conclude results of real-analytic solvability in differential complexes arising from locally integrable structures while in  \cite{CoJa20} the authors apply this tool to obtain certain global top degree solvability in CR and locally integrable hypocomplex structures.

The FBI transform is specially important to prove regularity properties in the hyperfunction setting where the  use of cuttoff functions is not tolerated. We proved that this generalized class of FBI transforms also characterize the wave-front set of  hyperfunctions regarding several subsheaves of the sheaf of hyperfunctions. Roughly speaking, given a hyperfunction and a subsheaf $\mathcal S$, the behaviour of the generalized FBI transform gives the directions (on the cotangent bundle) in which the hyperfunction does not belong to $\mathcal S$.

The novelty here is  to explore what we call the {\it microlocal decomposition} of an analytic functional, Definition~\ref{MLdecom},  which allowed us to obtain these equivalences  mainly as a consequence of the result for the sheaf of real-analytic functions. One can think it as a machinery to go from microlocal to local regularity often simplifying the study, that is, 
given a hyperfunction these techniques allow us to produce  a second hyperfunction and, independent of the regularity treated,  the microlocal regularity of the first can be establish from the local regularity of the second.

It is well-known that  every hyperfunction can be written as sum of boundary values of holomorphic functions, moreover, the way that one can write a hyperfunction as boundary value of holomorphic functions encodes its microlocal real-analyticity. One may use similar ideas  to discuss microlocal regularity outside the real-analytic realm.
Inspired by a remark after \cite[Theorem 8.4.15]{Hormander1}, we shall say that a hyperfunction is microlocal regular with respect to a subsheaf $\mathcal{S}$  in a given direction if it  is in $\mathcal{S}$ up-to boundary values of holomorphic functions defined in cones opposing the direction in question (see Section \ref{secMicroreghyp}). This notion agrees with the microlocal regularity in the real-analytic category.

We obtained a characterization of the wave-front set of a hyperfunction studying the growth of the generalized FBI transform when  we consider the regularity according  the sheaves $C^{\omega}, \Cinf, \mathcal{E}^{\{M\}}, \mathcal{D}', \mathcal{D}'^{\{M\}},$ respectively the sheaves of {\it real-analytic, smooth, Denjoy-Carleman } functions and the sheaves of {\it distributions and Denjoy-Carleman ultradistributions}. Let $\mathcal{S}$ be one of these sheaves then the microlocal regularity with respect to $\mathcal{S}$ can be classified in terms of estimates of the FBI transform  by the condition $\mathfrak{M}(\mathcal{S})$, see Section \ref{AFBIClassification}.  We avail ourselves of results in \cite{BH,  Furdos,HM} as ingredients which allowed us to use our technique to extend them to this richer framework.

As an application we proved the so-called elliptic regularity theorem for hyperfunctions with respect to several sheaves mentioned earlier, Theorem~\ref{Cor73}. This result says that if $P$ is a linear partial differential operator with real-analytic coefficients then the wave-front set of a hyperfunction $u$  with respect to $\mathcal S$ is contained in the wave-front set of $Pu$ (also with respect to $\mathcal S$) union with the characteristic set of $P$, \eqref{charP}, i.e.,
\begin{equation}
WF_{\mathcal S} (u)\subset WF_{\mathcal S} (Pu) \cup \Char P.
\end{equation}
Our proof also says that for any $\mathcal S$ subsheaf of the distributions that contains the sheaf of real-analytic functions as a subsheaf the elliptic regularity theorem holds for hyperfunctions if and only if it holds for distributions, see Corollary~\ref{Cor73???}.

Our work follows the approach to hyperfunctions present in \cite{CT} which follows  Martineau's ideas \cite{Martineau} and is more suitable to obtain explicit formulas  then Sato's original presentation (\cite{Sato1, Sato2}, see also \cite{kkk86}).
We recall that  Sato, in \cite{Satoellipticregularity}, introduced the notion of essential support (or singular spectrum) for hyperfunctions. Brós and Iagolnitzer, in \cite{BrosIagolnitzer}, used a variant of the Fourier transform, that we now call FBI transform, to characterize what they called the essential support in terms of the exponential decay of the FBI transform. Using pseudodifferential operators Hörmander introduced the notion of wave-front set. Finally, Bony  proved the equivalence between the Sato's essential support and the characterization via FBI transform of Brós and Iagolnitzer in \cite{Bony}.

Variants of the FBI transform gained importance recently, for example, the class introduced by   Berhanu and Hounie  in \cite{BH} was  used by them to characterize microlocal analyticity and microlocal smoothness of distributions. Recently, Hoepfner and Medrado~\cite{HM} studied these generalized FBI transforms (there denoted by FBI-BH transforms) and characterized the Denjoy-Carleman wave-front set of ultradistributions. Later, Fürdös~\cite{Furdos} also gave a characterization of the Denjoy-Carleman wave-front set. The main difference between these results is that Fürdös' result concerns a broader class of Denjoy-Carleman functions but deals only with distributions while Hoepfner and Medrado's work also take into account ultradistributions.

The elliptic regularity theorem is known in the context of hyperfunctions and real-analytic microlocal regularity as Sato's Theorem \cite{Satoellipticregularity}.
 Our proof is derived from the one of Sato's Theorem given in  Hörmander's book, \cite[Theorem 9.5.1]{Hormander1} which is an adaptation of the proof given in \cite{BonySchapira}.
To obtain this result for other sheaves we needed to go back to the paramount results in \cite{Boutet} regarding analytic pseudodifferential operators customizing them to our microlocal decomposition. This was the main ingredient in the proof and it is the content of the Appendix \ref{Appendix3}.

This paper is divided in two parts. In the first part we introduce a generalized FBI transform for analytic functionals. We start Section \ref{Preliminaries} recalling some basic machinery about approximations of entire functions and analytic functionals then, in Section \ref{secFBIanalytic}, we discuss some conditions on the  phase functions in order to extend Berhanu and Hounie's generalized FBI transforms to the framework of analytic functionals and we also prove an inversion formula for these FBI transforms. We conclude the first part providing our main example of phase functions.

The second part concerns microlocal regularity of hyperfunctions on several sheaves mentioned before.  Section \ref{secMicroreghyp} is destined to a quick overview on hyperfunctions followed by a discussion about boundary values of holomorphic functions and a method to define wave-front sets regarding any subsheaf of the sheaf of hyperfunctions. Then we  characterize, in Section \ref{secProfTheo2.4}, the analytic wave-front set of hyperfunctions. The proof of the characterization of the analytic wave-front mix ideas from Cordaro and Treves's approach for hyperfunctions  with the Berhanu and Hounie's techniques for the distribution set up and one of the outcomes is that the structure of our proof reduces the characterization of a wide range of wave-front sets to simple proposition, as one can see in Section \ref{AFBIClassification}. As an application we proved the elliptic regularity theorem for hyperfunctions considering the microlocal regularity with respect to several sheaves in Section \ref{secEllipticReg}. 

Throughout the paper we will use the notation $B_r(x)$ to denote the ball in $\R^n$ centered at $x\in\R^n$ and radius $r$ and $B_r^\C(x)$ to denote the ball in $\C^n$ centered at $x\in\C^n$ and radius $r$. To simplify and avoid confusion we will always assume that all open sets have smooth boundary, although in some cases this may not be necessary.

The authors would like to express their gratitude to Temple University's Mathematics Department for the hospitality during the time they spent there where part of this work was developed.

\section{Preliminaries}\label{Preliminaries}

Let $\mathcal{U}$ be an open subset of $\CN.$ The space of  holomorphic functions on $\mathcal{U}$ will be denote by $\mathcal{O}(\mathcal{U}).$ Given $A$ a subset of $\CN$ we denote by $A_{\delta}$ its $\delta$ neighborhood, i.e., 
\begin{equation}
A_{\delta}=  \{z \in \CN: \dist(z, A)< \delta\}.
\end{equation}
Let $\mathcal{K}$ be a compact subset of $\CN$ and denote by $\mathcal{O}(\mathcal{K})$ the space of germs of holomorphic functions on $\mathcal{K},$ in other words, $\mathcal{O}(\mathcal{K})$ is the inductive limit of the spaces $\mathcal{O}(\mathcal{K}_{\delta})$ when $\delta \lra 0.$
From now on $K$ will denote a compact subset of $\RN.$ It is well know that compact subsets of $\RN$ are Runge compact sets of $\CN$ which means that if $\mathcal{U}$ is an open neighborhood of $K$ in $\CN$ then every holomorphic function $h\in \mathcal{O}(\mathcal{U})$ can be approximate in a complex neighborhood of $K$ by entire functions.

There is a classical way to construct a family of entire functions that approximate $h$: choose an open subset  $W\subset\RN$ such that $K\subset W  \subset\subset \mathcal{U}$ and let $\chi$ be  the characteristic function of $W$, then we can define $H^{\epsilon}$ by
\begin{equation}
  \label{SeqConv}
  H^{\epsilon}(w)= \frac{1}{(2\pi)^{N}} \int_{\RN} \int_{\RN} e^{-\epsilon|\xi|^{2}} e^{i(x-w)\xi} \chi(x) h(x) \dd \xi \dd x.
\end{equation}
A proof of this result can be found in  \cite{CT}. It is an exercise to check that this approximation scheme also yields the following property:
\begin{equation}\label{aoured}
\begin{array}{c}
\text{there exist  $\delta_0>0$ such that for every fixed $\delta \in (0,\delta_0]$ and  $\tilde{\epsilon}>0$} 
\\[5pt]
\text{ we can find $\epsilon_0>0$ such that for all $\epsilon \in (0,\epsilon_0]$ it follows that}
\\[5pt]
\displaystyle  \sup_{K_\delta}|h-H^\epsilon|\leq \tilde{\epsilon}\sup_{W_{2 \delta}}|h|.
\end{array}
\end{equation}

We will now  recall the definition of analytic functionals. 

\begin{Def}[Analytic Functional]
Let $\Omega$ be an open subset of $\CN$. Consider $\mathcal{O}(\Omega)$ with its usual topology as a closed subspace of the space of continuous functions on $\Omega$. We say that $\mu$ is an analytic functional if it is a continuous linear functional on $\mathcal{O}(\Omega)$.  We denote by $\mathcal{O}'(\Omega)$ the space of all analytic functionals on $\Omega$.
\end{Def}

\begin{Def}[Analytic Functional carried by a compact set]\label{def:AF}
Let $\mathcal{K}$ be a compact subset of $\CN.$ We say that $\mu \in \mathcal{O}'(\CN)$ is  carried by $\mathcal{K}$ if, for every $\delta>0,$ there exists $C_\delta>0$ such that
\begin{equation}\label{deffuncanal}
  |\mu(h)|\leq C_\delta \sup_{z \in \mathcal{K}_\delta} |h(z)|, \quad \forall h \in \mathcal{O}(\CN).
\end{equation}
\end{Def}

 We are going to denote the space of analytic functionals on $\CN$ carried by $\mathcal{K}$ by $\mathcal{O}'(\mathcal{K})$. Let us point out that we are consciously abusing of the notation when we denote the space of analytic functionals carried by $\mathcal{K}$ by  $\mathcal{O}'(\mathcal{K}).$ This space can be identified with the dual space of $\mathcal{O}(\mathcal{K})$ whenever $\mathcal{K}$ is a Runge compact subset of $\CN$ and it will always happens in this text since we will work with compact subsets of $\RN$.

\section{The generalized FBI transform of analytic functionals}\label{secFBIanalytic}

\subsection{Good Phase Functions and the Generalized FBI Transform}

\begin{Def}[Good Phase Functions]\label{def:hpr9byt}
  Let $\Psi \in \mathcal{O}(\CN)$ and fix $\lambda>0$ a parameter. We will say that $\Psi$ has the {\it properties of a good phase function} (for $\lambda$) if it satisfies the following properties:
  \begin{enumerate}
  \item  For every $K\subset \RN$ compact set, every $\delta>0$ and every $\epsilon\in (0,1)$, there is $M>0$ such that
    \begin{align*}
      \int_{\RN}e^{-\epsilon |\xi|^{2}}\bigg(\int_{\RN}|\Psi\big(|\xi|^{\lambda}( \tau- w)\big)| |\xi|^{\lambda N}\dd \tau\bigg) \dd \xi<M, \quad \forall \,w \in K_{\delta}.
    \end{align*}
  \item There is a constant $L>0$ such that
    \begin{align*}
      \int_{\RN} \Psi(t^{\lambda}(\tau- w))t^{\lambda N} \dd \tau=1,\quad \forall\,w \text{ with } |\Im w|< L \ \text{ and } \forall \, t>0.
    \end{align*}
  \end{enumerate}
\end{Def}

We are now ready to introduce the FBI transform of an analytic functional.
\begin{Def}[Generalized FBI Transform] Let $K$ be a compact subset of $\RN$, $\lambda$ be a positive number and $\Psi \in \mathcal{O}(\CN)$.  We define the FBI transform (with respect to $\Psi$ and $\lambda$) of an analytic functional $\mu\in \mathcal{O}'(K)$  by
  \begin{equation}
    \label{FBIdef}
    \mathcal{F}_{\Psi, \lambda}\mu(\tau, \xi)= \mu_{w}\Big( e^{i(\tau- w) \xi} \Psi\big( |\xi|^{\lambda}(\tau- w)\big)\Big), \quad \forall \tau\in \CN,\forall \xi \in \RN.
    \end{equation}
\end{Def}

Let us discuss the regularity of $\mathcal{F}_{\Psi, \lambda}\mu$. Since $ e^{i(\tau- w) \xi} \Psi\big( |\xi|^{\lambda}(\tau- w)\big)$ is a continuous function of $(\tau, \xi)$ the same is true for $\mathcal{F}_{\Psi, \lambda}\mu$, thanks to \eqref{deffuncanal}.
In fact, $\mathcal{F}_{\Psi, \lambda}\mu$ is $\Cinf$ in $(\tau,\xi)$ for $\xi\neq 0$ and, for $\xi$ fixed, it is an entire function of $\tau$.

 In order to prove the claim we recall that $\mathcal{O}(K_\delta)$ is a closed subspace of $C(K_\delta)$ so every element of $\mathcal{O}'(K_\delta)$ can be identified with a Radon measure with compact support. Since $\mathcal{O}'(K) \subset \mathcal{O}'(K_\delta)$ we see that, for every $\mu\in \mathcal{O}'(K),$ there exists  a Radon measure $\nu$ with compact support in $K_{\delta}$ such that
\begin{align}\label{RadonMeasure}
  \mu(h) = \int_{K_\delta} h(w) \dd \nu(w).
\end{align}
Therefore \eqref{RadonMeasure} allow us to differentiate under the integral sign in \eqref{FBIdef}, proving  the claim. 

\medskip
 
Let $\mu$ be in  $\mathcal{O}'(K)$. Choose any  open bounded neighborhood $K\subset W\subset\subset \RN$ and consider $\chi$ to be the characteristic function of $W$. Define the family  $\mu_\epsilon$ of functions by the expression:
\begin{align}\label{FBI:inversion}
  \mu_\epsilon(x) &= \bigg( \int_{\RN} \int_{\RN} e^{i(x- \tau)\xi- \epsilon|\xi|^{2}} \mathcal{F}_{\Psi, \lambda}\mu(\tau, \xi)|\xi|^{\lambda N} \dd \tau \dd \xi\bigg) \frac{ \chi(x)}{(2\pi)^{N}}
\end{align}

For a fixed $\epsilon>0$, we can regard $\mu$ as a Radon measure in order to commute the integrals and the analytic functional in \eqref{FBI:inversion}, to obtain
\begin{align}\label{FBI:inversion1}
  \mu_\epsilon(x)= \mu_w \bigg( \int_{\RN} \int_{\RN} e^{i(x- \tau)\xi- \epsilon|\xi|^{2}} \Big( e^{i(\tau- w) \xi} \Psi\big( |\xi|^{\lambda}(\tau- w)\big)\Big)|\xi|^{\lambda N} \dd \tau \dd \xi\bigg) \frac{ \chi(x)}{(2\pi)^{N}}.
\end{align}
Now one can  check that, differentiating under the integral sign, the function
\begin{align}\label{ISENTIRE}
z \mapsto \int_{\RN} \int_{\RN} e^{i(z- \tau)\xi- \epsilon|\xi|^{2}} \mathcal{F}_{\Psi, \lambda}\mu(\tau, \xi)|\xi|^{\lambda N} \dd \tau \dd \xi
\end{align}
is an entire function, which further implies that $\mu_\epsilon$ is a real-analytic in $W$. 

\subsection{The inversion formula.}\label{InvFBI}

We are now ready to prove an inversion formula for these FBI transforms. This inversion formula is a powerful device that will be specially useful to study the microlocal regularity of a hyperfunction.  

\begin{Pro}[Inversion Formula] \label{fbibhif}
Let $W$ be an open subset of $\RN$, $K$ be a compact subset of $W$ and let $\chi$ be the characteristic function of $W.$ Then for a given analytic functional $\mu \in \mathcal{O}'(K)$ it follows that  $\mu_\epsilon$ given by \eqref{FBI:inversion} converges to $ \mu$ in $\mathcal{O}'(\overline{W}).$
\end{Pro}

\begin{proof}
Fix $h \in \mathcal{O}(\CN)$. We identify $\mu_\epsilon$ with an analytic functional and use \eqref{FBI:inversion1} to compute
\begin{align}\label{epoayre}
\mu_\epsilon(h)&= \int_{W} \mu_\epsilon(x)h(x)\dd x
\notag\\[5pt]
&=\int_{W}   \mu_w \bigg( \int_{\RN} \int_{\RN} e^{i(x- \tau)\xi- \epsilon|\xi|^{2}} \Big( e^{i(\tau- w) \xi} \Psi\big( |\xi|^{\lambda}(\tau- w)\big)\Big)|\xi|^{\lambda N} \dd \tau \dd \xi\bigg)  \frac{h (x)}{(2 \pi)^{N}}  \dd x 
\notag\\[5pt]
& =\mu_{w}\bigg(\int_{W}  \int_{\RN}  e^{i(x-w) \xi - \epsilon|\xi|^{2}} \int_{\RN} \Psi\big( |\xi|^{\lambda}(\tau- w)\big) |\xi|^{\lambda N} \dd \tau \dd \xi \frac{h(x)}{(2 \pi)^{N}} \dd x\bigg) 
\notag\\[5pt]
& =\mu_{w}\bigg(\int_{\RN}  \int_{\RN}  e^{i(x-w) \xi - \epsilon|\xi|^{2}}   \frac{\chi (x)}{(2 \pi)^{N}} h(x)\dd \xi \dd x\bigg) \notag\\[5pt]
&= \mu(H^{\epsilon}).
\end{align}
Now using the approximation scheme described in \eqref{aoured}, we conclude that
given $\tilde{\epsilon}>0$ and $\delta\in (0,\delta_0]$ we can find $\epsilon_0$ such that for each $\epsilon\in (0,\epsilon_0]$ it holds
\begin{align*}
  |(\mu-\mu_\epsilon)(h)|&= |\mu(h- H^\epsilon)|\\
&\leq C_\delta \sup_{K_\delta} |h-H^\epsilon|\\
&\leq  \tilde{\epsilon}C_\delta \sup_{W_{2\delta}} |h|.
\end{align*}
This proves that $\mu_\epsilon \lra \mu$ in $\mathcal{O}'(\overline{W})$.
\end{proof}

\subsection{Main example of the generalized FBI transform.}

Consider a homogeneous polynomial $p$ of degree $2k$ given by
\begin{equation}\label{ex:1}
p(z) = \sum_{ |\alpha|= 2k} a_\alpha z^\alpha, \quad \forall z\in \CN
\end{equation}
where $\alpha\in \ZN$ and $a_\alpha\in \R$. We also assume that there exist positive constants $c,C$ such that 
\begin{equation}\label{ex:2}
  c|x|^{2k} \leq p(x)\leq C|x|^{2k}, \quad \forall x\in \RN.
\end{equation}

It is easy to see that property \eqref{ex:2} can be extended to a conic complex neighborhood of $\RN \setminus\{0\}$ in the following way: there exist positive constants $c', C'$ and $\rho$ such that
\begin{equation}\label{DesParComp}
  c'|z|^{2k} \leq \Re p(z)\leq C'|z|^{2k}, \quad \forall z\in \Gamma_\rho,
\end{equation}
where $\Gamma_\rho:= \{ z= x+ i y\in\CN: |y| \leq \rho|x| \}.$ 

From now on we will consider the family of entire functions 
\begin{equation}\label{laie86t}
\Psi(z):=c_p e^{-p(z)},\qquad (c_p)^{-1}:=\int_{\RN} e^{-p(x)} \dd x.
\end{equation}
The next result shows that $\Psi$ satisfies the {\it properties of a good phase function} for $\lambda\in (0, 1/k)$, as described in Definition~\ref{def:hpr9byt}.

\begin{Lem}
Let $\Psi(z)$ defined by \eqref{laie86t} with $p(z)$ satisfying \eqref{DesParComp}. Then for all $\lambda\in (0, 1/k)$, $\Psi$ satisfies properties (1) and (2) from Definition~\ref{def:hpr9byt}.
\end{Lem}

\begin{proof}
We first prove the integrability condition (1) from Definition~\ref{def:hpr9byt}. First, %consider $K\subset \RN$ to be a compact set and $\delta>0$. %such that given $\epsilon>0$ we can find $M>0$ such
%\begin{align*}
%\int_{\RN} e^{-\epsilon t^{2}} \bigg( \int_{\RN} c_{p} |e^{-p(t^{\lambda}(\tau- w)}| |t|^{\lambda N} \dd \tau \bigg) \dd t<M.
%\end{align*}
%
note that
\begin{align}\label{ine1}
  c_{p}\int_{\RN} |e^{-p(|\xi|^\lambda(\tau-w))}| |\xi|^{ \lambda N} \dd \tau & = c_{p}\int_{\RN} e^{-\Re \{p(|\xi|^\lambda(\tau-w))\}} |\xi|^{ \lambda N} \dd \tau.
\end{align}
%
%the real part of $p$ gives us all the necessary information. 
Now, fix $K\subset \RN$ a compact set. For any $\delta>0$ there exists a constant $A=A(K,\delta,\rho)$ such that if $w \in K_{\delta}$ and  $\tau\in\RN$ are such that $|\tau- w|>A,$ then $\tau-w\in \Gamma_\rho$ (where $\Gamma_\rho$ is such that \eqref{DesParComp} holds).
Therefore, one can apply  \eqref{DesParComp} together the homogeneity $\Re \{p(|\xi|^\lambda(\tau-w))\}=  |\xi|^{2k\lambda} \Re\{ p(\tau- w)\}$ to obtain
\begin{equation}\label{eaipue}
  c' |\xi|^{2k\lambda} |\tau-w|^{2k} \leq \Re \{p(|\xi|^\lambda(\tau-w))\}, \quad \forall w \in K_{\delta},\, |\tau- w|>A.
\end{equation}
Thus, it follows from \eqref{eaipue} that  the right hand-side of \eqref{ine1} can be estimate by
\begin{align}\label{aioeudu}
c_{p}\int_{\RN} &e^{-\Re \{p(|\xi|^\lambda(\tau-w))\}} |\xi|^{ \lambda N} \dd \tau 
\notag\\[5pt]  
&\le c_{p}|\xi|^{ \lambda N} \bigg(\int_{|\tau-w|\leq A} e^{-\Re \{p(|\xi|^\lambda(\tau-w))\}}  \dd \tau
+ \int_{|\tau-w|> A} e^{-c' |\xi|^{2k\lambda} |\tau-w|^{2k}}  \dd \tau \bigg).
\end{align}
The second term in the right hand-side of \eqref{aioeudu}
can be bounded by
\begin{equation}\label{eaipue2}
c_{p} |\xi|^{ \lambda N}\int_{\RN} e^{- c' |\xi|^{\lambda}| \tau|^{2k}}  \dd \tau \le 
c_p \int_{\RN}  e^{-c'|\eta|^{2k}} \dd \eta=C.
\end{equation}
%
%which, after multiplication by $e^{-\epsilon t^{2}}$ it is an integrable function of $t$ for any $\epsilon>0$ 
%
For the first term in the right hand-side of \eqref{aioeudu}, consider 
\begin{equation*}\label{eaipue3}
S:= \sup \{-\Re p(\tau-w): w\in K_{\delta},\,  \tau \in \RN \, \text{ with }\, \dist(\tau,K_\delta)<A \}.
\end{equation*}
Thus one can bound the first integral in the right hand-side of \eqref{aioeudu} by
\begin{equation}\label{eaipue4}
 c_{p}\int_{|\tau-w|\leq A} e^{-\Re \{p(|\xi|^\lambda(\tau-w))\}} |\xi|^{ \lambda N} \dd \tau
 \le  c_p m(B_A) e^{ |\xi|^{2k\lambda} S} |\xi|^{ \lambda N} 
\end{equation}
where $m(B_A)$ stands for the Lebesgue measure of the radius $A$ ball in $\RN$.
Therefore, integrability condition is a consequence of \eqref{aioeudu}, \eqref{eaipue2} and \eqref{eaipue4}:
\begin{align}\label{aioeudua}
\int_{\R}e^{-\epsilon |\xi|^{2}}&\bigg(\int_{\RN} |e^{-p(|\xi|^{\lambda}(\tau-w))}| |\xi|^{\lambda N}\dd \tau\bigg) \dd \xi 
\notag\\[5pt]  
&\le C\int_\R e^{-\epsilon |\xi|^2} \dd \xi + c_p m(B_A) \int_\R e^{-\epsilon |\xi|^2 + |\xi|^{2k\lambda} S} |\xi|^{ \lambda N} \dd \xi  <\infty
\end{align}
since $\lambda \in (0,1/k)$.

Moving on, we are now going to prove  property (2) from Definition~\ref{def:hpr9byt}. To do so, consider the function
\begin{equation}\label{diausny}
 H(w):= c_p\int_{\RN}  e^{- p(t^{\lambda} (\tau- w))} t^{\lambda N} \dd \tau, \quad w \in \CN.
\end{equation}
The  arguments above can be used to prove that $H(w)$ is an entire function. 
When $w\in \RN$ we can perform the change of variables $\tau \lra \tau/t^\lambda+ w$  and conclude that $H|_{\RN} \equiv 1$. Therefore $H\equiv1$ on $\CN$.
\end{proof}

\section{Microlocal regularity of hyperfunctions}\label{secMicroreghyp}

Now we start the second part of this work where we will use the FBI transform to analyze microlocal regularity of hyperfunctions on several functional analytic spaces. Since the way to manipulate hyperfunctions is quite different from the way we treat distributions we will summarize some standard definitions and results. Our main reference for this subject is Cordaro and Treves' book~\cite{CT}.

\subsection{A review of hyperfunctions}
Let $\mathcal{B}$ be the sheaf of (germs of) hyperfunctions on $\RN$. We recall that the restriction of a hyperfunction to  a bounded open subset  $U$ of $\RN$ can be identified with a class of analytic functionals $[\mu] \in \mathcal{O}'(\overline{U})/\mathcal{O}'(\del U)$. In other words, if $\mu_1$ and $\mu_2$ are two representatives of a hyperfunction $u\in \mathcal{B}(U),$ then $\mu_1 - \mu_2 \in \mathcal{O}'(\del U).$ Moreover, if $V\subset\subset U\subset\subset \RN$, $u \in \mathcal{B}(U)$ and $\mu\in \mathcal{O}'(\overline{U})$ is one representative of $u$, then there are $\nu \in \mathcal{O}'(\overline{V})$ and $\lambda \in \mathcal{O}'(\overline{U}\setminus V)$ such that $\mu = \nu+ \lambda$ and $u|_{V}= [\nu].$ If $\Omega$ is any open set of $\RN$, then  any linear differential operator with real-analytic coefficients in $\Omega$ defines a linear operator acting on $\mathcal{B}(\Omega)$.

\subsubsection{Boundary values in the sense of hyperfunctions.}
 Let $\Gamma$ be an open cone in $
\RN\setminus \{0\}$ and let $\delta>0$ be an arbitrary number. If $U$ is an open subset of $\RN$ we define the wedge
\begin{equation*}
  \mathcal{W}_\delta(U;\Gamma)= \{ x+ iy: x\in U, y\in \Gamma \textrm{ with } |y|< \delta\}.
\end{equation*}

Given $ g \in \mathcal{O}(\mathcal{W}_\delta(U;\Gamma))$ and an open subset $V\subset\subset U$ whose boundary is smooth we can define for every $y \in \Gamma$, $|y|< \delta$, the analytic functional $\mu_{g,V}^{y}$ by
\begin{equation}\label{rifenrhct}
  \langle \mu_{g, V}^{y} , h \rangle= \int_{V+ iy} g(z) h(z) \dd z, \quad \forall h \in \mathcal{O}(\CN).
\end{equation}

The next remark is the content of   \cite[Theorem II.1.1]{CT}.

\begin{Rem}\label{faluh}  There is an analytic functional $\mu_{g, V} \in \mathcal{O}'(\overline{V})$ that satisfies the following properties:
\begin{itemize}
\item [(i)]To every neighborhood of $\mathcal{U}$ of $\del V$ in $\CN$ there is $\epsilon$, $0< \epsilon < \delta$ such that $\mu_{g,V}- \mu_{g,V}^{y} \in \mathcal{O}'(\mathcal{U})$ for every  $y \in \Gamma$ with $|y|<\epsilon$;
\item [(ii)] Moreover, if  $W$  is an open set with smooth boundary and $V\subset \subset W\subset \subset U$ then $\mu_{g, W}- \mu_{g, V} \in \mathcal{O}'(\overline{W}\setminus V)$.
\end{itemize}
\end{Rem}

We are now ready to define boundary values in the sense of hyperfunctions.

\begin{Def}[Boundary Values]\label{def:BV}
  Given $ g \in \mathcal{O}(\mathcal{W}_\delta(U;\Gamma))$, we define the boundary value of $g$ in $V$ as the hyperfunction in $V$ whose representative is $\mu_{g,V}$ given in  item $(i)$ of Remark~\ref{faluh}. We are going to denote this hyperfunction by $\bb g$.

Also,  item $(ii)$ of Remark~\ref{faluh} allows us to define $\bb g$ as the hyperfunction in $U$ whose restriction to any bounded subset $V$ of $U$ is represented by $\mu_{g, V}$.
Note that the property described in $(i)$ of  Remark~\ref{faluh} does not need to hold for the representatives of $\bb g$ in $U$, therefore, to avoid confusion, we will use $\nu_g$ to refer to a representative of $\bb g$ in $U$. 
\end{Def}

One can check that if $a\in C^{\omega}(U)$ and $\tilde{a}$ stands for its complexification to $U_\delta\subset\CN$, then $a \bb g= \bb (\tilde{a} g)$ where $g \in \mathcal{O}(\mathcal{W}_{\delta}(U;\Gamma)).$

\subsubsection{Analytic microlocal regularity for hyperfunctions}

It is known that every hyperfunction can be written as sum of boundary values of holomorphic functions (see Theorem II.3.1 of \cite{CT}). Moreover, the microlocal regularity of a hyperfunction can be defined using boundary values of holomorphic functions.

 For future reference, let us fix the following notation: given $\xi_0 \in \RN\setminus \{0\}$ and a cone $\Gamma\subset \RN\setminus \{0\}$ we will say that $\Gamma$ is an {\it opposite cone} for $\xi_0$ or that $\Gamma$ is {\it opposing the direction  $\xi_0$} if
\begin{align}\label{laliuwe}
  \xi_0\cdot y<0\ \textrm{ for all }\ y \in \Gamma.
\end{align}

% \begin{Def}[Microlocal Analyticity and the Analytic Wave-Front Set]\label{kasgfbyu}
%   We shall say that a hyperfunction $u$ in $\RN$ is microlocally real-analytic at $(x_0, \xi_0) \in \RN\times (\RN\setminus\{ 0\})$ if there exist an open neighborhood $U$ of $x_0$ and finitely many open convex cones $\Gamma_j \subset \RN\setminus \{0\}$, $j=1, \ldots, \ell,$ satisfying  the following properties:
%   \begin{enumerate}
%   \item for every $j=1, \ldots, \ell$, $\Gamma_j$ is an opposite cone for $\xi_0$;
%   \item there exist $F_j \in \mathcal{O}(\mathcal{W}_\delta(U;\Gamma_j)), (\delta>0, j=1,\ldots, \ell)$ such that $u= \sum_{j=1}^{\ell} \bb F_j.$
%   \end{enumerate}

%   The set of points $(x,\xi) \in \RN\times (\RN\setminus\{0\})$ in  which $u$ is not microlocally analytic will be called the analytic wave-front set of $u$ and it will be denoted by $WF_{C^\omega}(u).$
% \end{Def}

 Let $\mathcal{S}$ be a sheaf that can be embedded in the sheaf of hyperfunctions, i.e., there is an injective  homomorphism of sheaves $\iota: \mathcal{S} \hookrightarrow \mathcal{B}$. We will say that a hyperfunction $u\in \mathcal{B}(\RN)$ belongs to $\mathcal{S}(U)$ if we can find $v\in \mathcal{S}(U)$ such that $u|_U= \iota(v).$

We are going to introduce a concept of { \it microlocal regularity with respect to $\mathcal{S}$.}
\begin{Def}[Microlocal $\mathcal{S}$-regularity]\label{alsiu980}
    We shall say that a hyperfunction $u$ in $\RN$ is microlocally $\mathcal{S}$-regular at $(x_0, \xi_0) \in \RN\times (\RN\setminus\{ 0\})$ if there are $U$ an open neighborhood of $x_0$ and finitely many open convex cones $\Gamma_j \subset \RN\setminus \{0\}$, $j=1, \ldots, \ell,$ with the following properties:
  \begin{enumerate}
  \item  for every $j=1, \ldots, \ell$, $\Gamma_j$ is an opposite cone for $\xi_0$;
  \item there exist $F_j \in \mathcal{O}(\mathcal{W}_\delta(U;\Gamma_j)), (\delta>0, j=1, \ldots, \ell)$ such that $u- \sum_{j=1}^{\ell} \bb F_j$ belongs to $\mathcal{S}(U)$.
  \end{enumerate}

  The set of points $(x,\xi) \in \RN\times (\RN\setminus\{0\})$ at which $u$ is not microlocally $\mathcal{S}$-regular will be called the $\mathcal{S}$-wave-front set of $u.$ The $\mathcal{S}$-wave-front set of a hyperfunction $u$ will be denoted by $WF_{\mathcal{S}}(u).$
\end{Def}

We recall that there is a natural way to embed the sheaf of distributions into the sheaf of hyperfunctions and, consequently, the same is true for the subsheaves of $\mathcal{D}'$, e.g., the sheaf of real-analytic functions, sheaf of smooth functions, etc. For every $x_0 \in \RN$ we denote by $\mathcal{B}_{x_0}$ and $\mathcal{D}'_{x_0}$ the spaces of germs of hyperfunctions and distributions at $x_0$, respectively.
We can identify ${\bf u} \in \mathcal{D}'_{x_0}$ with an element of $\mathcal{B}_{x_0}$ in the following way: let $(U, u)$ be a representative of ${\bf u}$ and choose $\varphi \in \Cinf_{c}(U)$ such that $\varphi=1$ in a neighborhood of $x_0$, then $\varphi u$ defines an analytic functional $\nu_{\varphi u}$ by the formula:
\begin{align}\label{asouejri}
    \nu_{\varphi u}(h) = u( \varphi h) , \quad \forall h \in \mathcal{O}(\CN).
  \end{align}
One can check that the germ of hyperfunction defined by $\nu_{\varphi u}$ is independent of the representative of ${\bf u}$ and  the chosen cutoff function. 
Let $U$ be a bounded open subset of $\RN.$ This embedding allow us to identify every $u\in \mathcal{D}'(U)$ with a class of analytic functionals $[\nu_u]\in \mathcal{O}'(U)/\mathcal{O}'(\del U)$ with the property that any representative $\nu_u$ satisfies $\nu_u- \nu_{\varphi u} \in \mathcal{O}'(\overline{U}\setminus V)$ whenever $V\subset \subset U$, $\varphi \in \Cinf_c(U)$ satisfies $\varphi =1$ in $V$ and $\nu_{\varphi u}$ is defined by \eqref{asouejri}.

  % When $\mathcal{S}$ is the sheaf of analytic functions, $C^{\omega}$, then the concepts of microlocal analyticity given in 
  % Definition~\ref{kasgfbyu} and microlocal $C^{\omega}$-regularity provided in Definition~\ref{alsiu980} agree.

\subsubsection{The generalized FBI transform in the study of microlocal regularity}

From now on we are going to consider a particular subclass of the generalized FBI transform. We pick $\lambda= 1/2k$ and choose $p$ satisfying \eqref{ex:1} and \eqref{ex:2}. Consider $\Psi(z)= c_p e^{-p(z)}$ as in \eqref{laie86t} and 
the generalized FBI transform
\begin{equation}\label{pe84n}
  \mathcal{F}_{p}\mu(\tau, \xi):=\mathcal{F}_{\Psi, 1/(2k)}(\tau, \xi)= c_p\mu_w\Big( e^{i(\tau-w)\xi -|\xi| p(\tau-w)}\Big).
\end{equation}

Let $\mu\in \mathcal{O}'(\overline{U})$. Given an open subset  $V\subset\subset U$  and a cone $\Gamma\in \RN$  we will say that $\mu$ satisfies the condition $\mathfrak{M}(C^{\omega})$ of {\it microlocal real-analytic regularity} in $V\times \Gamma$ if there exist constants $c_1, c_2>0$ such that
  \begin{equation}
  	\tag{$\mathfrak{M}(C^{\omega})$}
    |\mathcal{F}_{p}\mu(\tau, \xi)|\leq c_1 e^{-c_2|\xi|}, \quad \forall (\tau, \xi) \in V\times \Gamma.
  \end{equation}

Let $u \in \mathcal{B}(\RN)$ and consider the following two regularity conditions for the FBI transform of representatives of $u$:
\begin{itemize}
\item [$(\mathfrak{R}_1)$] There exists an open neighborhood $U$  of $x_0$ in $\RN$ such that, for every open neighborhood $V\subset\subset U$ of $x_0$ and every  representative $\mu \in \mathcal{O}'(\overline{U})$ of $u$ in $U$, there are a conic neighborhood $\Gamma$ of $\xi_0$ in $\RN$ such that $\mu$ satisfies condition $\mathfrak{M}(C^{\omega})$ in $V\times \Gamma.$
\item [$(\mathfrak{R}_2)$]
  There exist  open neighborhoods $U, V$  of $x_0$ in $\RN$, $V\subset\subset U,$ a  representative $\mu \in \mathcal{O}'(\overline{U})$ of $u$ in $U$, a conic neighborhood $\Gamma$ of $\xi_0$ in $\RN$ such that $\mu$ satisfies condition $\mathfrak{M}(C^{\omega})$ in $V\times \Gamma.$
\end{itemize}

The following theorem states that conditions $(\mathfrak{R}_1)$ and $(\mathfrak{R}_2)$ are both equivalent to microlocal real-analyticity.
\begin{Teo}\label{Theo2.4}
  Let $u \in \mathcal{B}(\RN)$, $x_0 \in \RN$ and $\xi_0 \in \RN$. Then $(x_0, \xi_0)\notin WF_{C^\omega}(u)$ if and only if $(\mathfrak{R}_1)$ holds at $(x_0, \xi_0)$ if and only if $(\mathfrak{R}_2)$ holds at $(x_0, \xi_0)$.
\end{Teo}

 We will prove Theorem~\ref{Theo2.4} in Section~\ref{secProfTheo2.4}. It is a classical result that the  FBI transform can be used to characterize the analytic wave-front set of a hyperfunction. Theorem~\ref{Theo2.4} says that this characterization also holds if one considers these generalized FBI transforms.
Our goal is to prove several versions of Theorem~\ref{Theo2.4} for different many other sheaves that are common in the literature, e.g. smooth functions, distributions, Gevrey classes, Denjoy-Carleman  classes  and  ultradistributions, see Section \ref{AFBIClassification}.

\subsection{Invariance of the FBI estimates by change of representatives}

Before we focus in proving Theorem~\ref{Theo2.4}, we need to check that the decay of the generalized FBI transform is an invariant of the hyperfunction, i.e., if we change the representative we obtain the same type of decay.
Directly from the definition of  analytic functional carried by $K$ we obtain an estimate for $\mathcal{F}_{p}\mu(\tau, \xi)$ as follows: for every $\delta>0$ there exists a positive constant $C_\delta$ such that
\begin{align}\label{ib6f}
   |\mathcal{F}_{p}\mu(\tau, \xi)|\leq  C_\delta c_p \sup_{w\in K_{\delta}} e^{\Re \{i(\tau-w)\xi -|\xi| p(\tau-w)\}}.
\end{align}

  Note that $\Re\{i(\tau- w)\xi- |\xi|p(\tau-w)\}<0$ when $w \in K$, so given $L>0$ we can find $\delta>0$ such that $\Re\{i(\tau- w)\xi- |\xi|p(\tau-w)\}\leq L|\xi|$ for $\tau \in \RN$ and $w\in K_\delta$, thus we conclude that,
\begin{equation}\label{SEMPREVALE}
\begin{array}{c}
\text{for each constant $L>0$, there is $C_L>0$ such that}\\[5pt]
    |\mathcal{F}_{p}\mu(\tau, \xi)|\leq  C_L  e^{L|\xi|}, \quad (\tau,\xi)\in \RN\times \RN.
  \end{array}
\end{equation}

We can refine this statement when $\tau$ is far from $K.$ 

\begin{Pro}\label{Invariance}
Let $\rho$ be a positive number such that \eqref{DesParComp} holds. Fix  disjoint compact subsets $K, K'\subset\RN$. Then there exist constants $a>0$ and $\delta>0$ such that
\begin{equation}\label{oasiy7t}
   |\mathcal{F}_{p}\mu(\tau, \xi)|\leq  C_\delta c_p  e^{-a|\xi|}, \quad \forall\, \tau \in K',\, \xi \in \RN\, \text{ and }\, \mu\in \mathcal{O}'(K_\delta)
\end{equation}
where $C_\delta$ depends on $\mu\in \mathcal{O}'(K_\delta)$ according with \eqref{deffuncanal}.
\end{Pro}
%%%A CONSTANTE TROCA COM O REPRESENTANTE!!!!!
\begin{proof}
All we have to do is to prove an estimate for the exponential in the right hand-side of \eqref{ib6f}:
\begin{equation*}
\Re \{i(\tau-w)\xi -|\xi| p(\tau-w)\}=  \Im w \cdot \xi - |\xi|\Re \{ p(\tau-w) \}
\end{equation*}
when $w \in K_{\delta}$, $\tau\in K'$ and for a $\delta>0$ to be chosen. % satisfying $\rho\geq \delta >0$.
%We start noticing that
%%
%\begin{align*}
%  \Re \{i(\tau-w)\xi -|\xi| p(\tau-w)\}&=  \Im w \xi - |\xi|\Re \{ p(\tau-w) \}.                                                          
%\end{align*}
%
Let $0<\sigma\le\tfrac12 \dist(K,K')$ then
when $\tau \in K'$ and $w \in K_{\delta}$ we  have $|\tau- \Re w|>\sigma$. 
Now we choose $\delta<\rho\sigma.$ Therefore,  it follows  from the choice of $\delta$ that
\begin{equation}\label{ajnrf}
|\Im(\tau- w)|= |\Im w|\le \delta<\rho\sigma \leq \rho |\tau- \Re w|, \quad \forall \tau \in K', \, \forall w \in K_{\delta}.
\end{equation}
Thus, as a consequence of \eqref{ajnrf} and  \eqref{DesParComp}, we obtain
\begin{align}
  \Re \{i(\tau-w)\xi -|\xi| p(\tau-w)\}&\leq   |\xi|(\delta - c'|\tau-w|^{2k})\nonumber\\
  &\leq -|\xi|(c'\sigma^{2k}-\delta), \label{finalinequality46}
\end{align}
for every $\tau \in K'$ and $ w \in K_{\delta}$.
Diminishing $\delta$, we can assume that $a:= c'\sigma^{2k}-\delta $ is a positive number. Thus, inequality \eqref{oasiy7t} follows from \eqref{ib6f} and \eqref{finalinequality46}. 
\end{proof}

As a consequence, if $\mu_1, \mu_2 \in \mathcal{O}'(\overline{U})$ are two representatives of the same hyperfunction, then their difference, $\mu_1-\mu_2$ is an analytic functional carried by $\del U.$ Given $\tau_0 \in U$, we  choose a neighborhood $V\subset\subset U$ of $\tau_0$ in $\RN$. Thus, Proposition~\ref{Invariance} with $K = \partial U$ and $K'=\overline V$, implies that there are positive constants $C$ and $a$ such that
\begin{align*}
  |\mathcal{F}_{p}(\mu_1-\mu_2)(\tau, \xi)| \leq  C e^{-a|\xi|}, \quad \forall \tau \in V, \xi \in \RN.
\end{align*}   

With this inequality in hand it will be readily to verify that the FBI estimates are invariants of the hyperfunction in the sense that changing the representative of the hyperfunction only affects the estimates by a term that decays exponentially.

We will conclude this section with an observation on the microlocal regularity conditions $(\mathfrak{R}_1)$ and $(\mathfrak{R}_2)$. Let $x_0 \in \RN $ and $\xi_0\in \RN\setminus \{0\}$. Assume that $u \in \mathcal{B}(\RN)$ and $\mu\in \mathcal{O}'(\overline{U})$ is a representative of $u$ in $U$ with the property that there are an open neighborhood $V\subset\subset U$ of $x_0$, $\Gamma$ a conic neighborhood of $\xi_0$ and constants $c_1, c_2>0$ such that
  \begin{align*}
    |\mathcal{F}_{p}\mu(\tau, \xi)|\leq c_1 e^{-c_2|\xi|}, \quad \forall (\tau, \xi) \in V\times \Gamma.
  \end{align*}
Let $W\subset U$ and consider $\nu$ the restriction of $\mu$ to $W$. We have that $\mu-\nu\in \mathcal{O}'(\overline{U} \setminus W)$ thus, Proposition~\ref{Invariance} guarantees that if $V'\subset\subset V\cap W$, then there are $C_1,C_2>0$ such that
\begin{align*}
  |\mathcal{F}_p\nu(\tau, \xi)|\leq C_1e^{-C_2|\xi|}, \quad \forall(\tau, \xi) \in V'\times \Gamma.
\end{align*}
This means that if either $(\mathfrak{R}_1)$ or $(\mathfrak{R}_2)$ holds for $\mu$ a representative of $u$ in $U$, then the same property holds for the restriction of $\mu$ to any open neighborhood of $x_0$ contained in $U$.

\section{Proof of Theorem~\ref{Theo2.4}}\label{secProfTheo2.4}

We devote this section to prove Theorem~\ref{Theo2.4}. We will start with estimates for the FBI transform of boundary value of holomorphic functions in opposite cones. With this estimates in hand, we can generalize Theorem~\ref{Theo2.4} to some  subsheaves of $\mathcal{B}$. When  $\mathcal S$ is one of these subsheaves we can  exploit the concept of microlocal regularity with respect to $\mathcal{S}$ and reduce the proof to obtain  estimates for the FBI transform of sections of $\mathcal{S}$.

The second part of the proof explores the FBI inversion formula to obtain a general decomposition of a representative of a hyperfunction as sum of analytic functionals. In this sum we can isolate one analytic functional that carries the crucial information regarding the regularity in the studied direction (the other analytic functional are representatives of boundary value of holomorphic functions in  cones opposing the given direction). The regularity of the isolated part is characterized by estimates of its FBI transform and, therefore, there is a dictionary that associate estimates of the FBI transform with microlocal regularity regarding a subsheaf.  

\subsection{Boundary values of holomorphic functions in opposite cones}

We start proving that microlocal analyticity implies an exponential decay of the generalized FBI transform. The first step is the case where the hyperfunction is the boundary value of a holomorphic function.

\begin{Pro}\label{FBIBHBV}
Let $\xi_0 \in \RN$ and $\Gamma$ be a convex cone in $ \RN\setminus \{0\}$ which is opposite to $\xi_0$. Then given $x_0\in \RN$ and  open neighborhoods $V\subset\subset U \subset\subset \RN$ of $x_0$,  there is an open cone $\Gamma'\subset \RN$ with $\xi_0 \in \Gamma'$ such that, if  $\nu_g$ is an analytic functional that represents the boundary value of $g \in \mathcal{O}(\mathcal{W}_{\delta_0}(U;\Gamma))$, $0<\delta_0<1$, then  there exist positive numbers $C,R$ for which the inequality
  \begin{equation}\label{DesRA}
    |\mathcal{F}_{p}\nu_{g}(\tau, \xi)|\leq C e^{-R |\xi|}, 
  \end{equation}
  holds for every $\tau \in V$ and $\xi\in \Gamma'$.
\end{Pro}
%%% A CONSTANTE TROCA COM O REPRESENTANTE!!!!
\begin{proof}
  For a better understanding, we will divide the proof in three steps.

\noindent{\bf Step 1: Some reductions.} 
Let $g \in \mathcal{O}(\mathcal{W}_{\delta_0}(U;\Gamma))$ and let $\nu_g$ be an analytic functional in $U$ that represents $\bb g$. Fix $V\subset\subset W\subset\subset  U$ neighborhoods of $x_0$, $W$ with smooth boundary. 
By Definition~\ref{def:BV}, the restriction of $\bb g$ to $W$ is represented by $\mu_{g, W}$ and consequently $\nu_g- \mu_{g,W} \in \mathcal{O}'(\overline{U}\setminus W)$. By Proposition \ref{Invariance}, it will be enough to prove inequality \eqref{DesRA} for $\mu_{g,W}$.

Choose $\tilde{\delta}$ small such that $\tilde{d}:= d \big((\del W)_{\tilde{\delta}}, V\big)>0$. By the definition of boundary values, Definition~\ref{def:BV}, we can find $\epsilon>0$ such that $\mu_{g,W}-\mu_{g,W}^y \in \mathcal{O}'((\del W)_{\tilde{\delta}})$ for every $y\in\Gamma$ with $|y|<\epsilon$. 
Again, thanks to Proposition \ref{Invariance}, it is sufficient to prove inequality \eqref{DesRA} for $\mu_{g,W}^{y}$. From now on, we will fix the direction  $y$, i.e., we may  reduce its norm but we will not change its direction.

In what follows we note that using linear algebra arguments (see Appendix \ref{AppA1} applied for $\xi_0/|\xi_0|$) the following is true:
\begin{enumerate}
\item[(i)] We can find $\tilde{\Gamma}$ a conic neighborhood of $y$ such that $\tilde{\Gamma}\subset \Gamma$. 
\item[(ii)] Since $\xi_0 \cdot y<0$ there is an open cone $\Gamma'$ in $\RN\setminus\{0\}$ with $\xi_0 \in \Gamma'$ such that $\xi\cdot y< 0$ for all $\xi \in \Gamma'$.
\item[(iii)] There is $\lambda<1$ with the property that if $\xi \in \Gamma'$ is such that $|\xi| \leq \lambda |y|$ then $ y - \xi$ belongs to $K$ a compact subset of $\tilde{\Gamma}$ ---it is important to point out that if we reduce the norm of $y$ then the latter holds with the same $\lambda,$ $\Gamma'$ and $y-\xi$ lying in a smaller compact subset of $\tilde{\Gamma}$.
\end{enumerate}

\noindent{\bf Step 2: Estimates on the phase of the FBI transform.} 
From now on, the pair $(\tau, \xi) \in V\times \Gamma'$ is arbitrary, we fix $\rho\in (0,1)$ such that the inequality \eqref{DesParComp} holds and choose $\delta\in (0,\delta_0)$. Since $p(0)=0$ there exists $\sigma\in (0,1)$ such that 
\begin{equation}\label{oaihe}
|\Re p(z)|\leq 1 \ \text{ for every } \  z \in B_\sigma^{\C}(0).
\end{equation}
 Now, if $\tau, x, a \in \RN$ are such that $|\tau- x|\leq \delta \sigma/2$ and $|a|\leq \sigma/2$ then
$|(\tau- x)/\delta +ia| \leq \sigma$ and we can use the homogeneity of $p(z)$ together with \eqref{oaihe}, to obtain
\begin{align}\label{lqwugtt}
  |\Re p(\tau- x+ i \delta a)|= \delta^{2k} |\Re p((\tau- x)/\delta +i a)|\leq \delta^{2k},\quad\forall\, x\in B_{\delta\sigma/2}(\tau),\, |a|\leq \sigma/2.
\end{align}
Assume, without loss of generality,  that $|y|= \delta\rho\sigma/4$. Let $r= \delta \sigma/2$ and  assume that $2r< \tilde{d}$ (shrinking $\delta$ if necessary) in order to have $B_{2r}(\tau)\subset W$ for every $\tau \in V$.

Choose $\phi \in \Cinf_c(B_{2r}(0))$ such that $0\leq \phi\leq \rho\lambda \sigma /4$, $\phi(x)\equiv  \rho\lambda \sigma /4$ in $B_r(0)$ and, for each $\tau\in V,$ define $\varphi\in \Cinf_{c}(W)$ as $\varphi(x)= \phi(x+\tau)$. Although the function $\varphi$ depends on $\tau$, all the inequalities obtained bellow will hold uniformly for any $\tau\in V$.
Then, by the construction of the compact set $K$ we see that
\begin{equation}\label{lfhr7464}
y - \delta \varphi(x) \xi/|\xi|\in K\quad\text{  for all }\, \xi\in \Gamma', \text{ and } \, x \in W.
\end{equation}

$\bullet$ If $x \in B_r(\tau),$ then $\varphi(x)= \rho\lambda \sigma /4< \sigma/4$ and we can use inequality  \eqref{lqwugtt} to obtain that % $| (\tau- x)/\delta -iy/\delta+ i \varphi_\tau(x)\xi/|\xi||\leq \sigma$ and 
$\big|\Re p\big(\tau- x-iy + i\delta\varphi(x)\xi/|\xi|\big)\big|\leq \delta^{2k}$ and, consequently, 
\begin{align}\label{29458yh}
  \delta \varphi(x) + \Re p\big(\tau- x -iy + i\delta\varphi(x)\xi/|\xi|\big)\geq \frac{\rho\delta \lambda \sigma }{4} -\delta^{2k}, \quad \forall\, x \in B_r(\tau), \,\xi\in\RN.
\end{align}
By diminishing $\delta$ if necessary we can assume that $c:= \frac{\rho\delta \lambda \sigma }{4} -\delta^{2k}>0$
and rewrite inequality \eqref{29458yh} as
\begin{align}\label{29458yha}
  \delta \varphi(x) + \Re p\big(\tau- x -iy + i\delta\varphi(x)\xi/|\xi|\big)\geq c, \quad \forall\, x \in B_r(\tau), \,\xi\in\RN.
\end{align}

$\bullet$ If $x\in W \setminus B_r(\tau)$, then
\begin{align}\label{87tw4}
  \bigg|y-\delta\varphi(x)\frac{\xi}{|\xi|}\bigg| \leq \frac{\rho\delta\sigma}{4}+ \frac{ \lambda\rho\delta\sigma}{4}= \frac{(1+\lambda)}{2} \rho r\leq \rho|\tau-x|.
\end{align}
With \eqref{87tw4} in mind, we are allowed to use \eqref{DesParComp} to obtain the following inequality
\begin{equation}\label{;.pro,jvypor}
\Re p(\tau-x -iy+ i \delta\varphi(x)\xi/|\xi|)\geq c' |\tau-x|^{2k}\geq  c'r^{2k}, \quad \forall\, x \in W\setminus B_r(\tau), \,\xi\in\RN.
\end{equation}
%
%Summarizing, when $y \in \Gamma_{\delta_0}$ and \red{$|y|<\delta \rho \sigma/4$*see above} it follows that
%\begin{equation*}
%\left\{
%\begin{array}{ll}
%  \delta \varphi_\tau(x) + \Re p(\tau-x-iy + i\delta\varphi_\tau(x)\xi/|\xi|)\geq \frac{\rho\delta\lambda \sigma}{4}- \delta^{2k} & \textrm{ if } x \in B_{r}(\tau); \\
%\delta \varphi_\tau(x) + \Re p(\tau-x-iy + i\delta\varphi_\tau(x)\xi/|\xi|)\geq r^{2k}c/2 & \textrm{ if } x \in W \setminus B_r(\tau). 
%\end{array}
%\right.  
%\end{equation*}
%
%We can choose $\delta$ small enough to assures that $\rho\delta\lambda \sigma/4-\delta^{2k}>0$. 
%
Let $R:= \min\{c, c'r^{2k}\}>0$ thus, it follows from \eqref{29458yha} and \eqref{;.pro,jvypor}, that
\begin{align}\label{;pocdr}
\delta \varphi(x) + \Re p(\tau-x- iy+ i \delta \varphi(x)\xi/|\xi|)\geq R,\quad \forall\,x \in W, \,\xi\in\RN.
\end{align}

\noindent{\bf Step 3: FBI transform estimate.}
Moving on, recall that the generalized FBI transform of the analytic functional $\mu_{g,W}^{y}$ is given by 
\begin{equation}\label{akeur}
\mathcal{F}_{p}\mu_{g,W}^{y}(\tau, \xi)=  c_{p}\int_W g(x+i y)  e^{i(\tau-x-iy)\xi-|\xi| p(\tau-x-iy)}\dd x
\end{equation}
We will now deform the contour of integration on the integral in the right hand-side of \eqref{akeur}, from $W$ to its image under the map 
\begin{equation}\label{loerihn}
\theta(x)= x - i\delta \varphi(x)\xi/|\xi|.
\end{equation}
and we  exploit the  fact that the  $N$-form $g(z)e^{i(\tau-z)\xi-|\xi|p(\tau-z)}\dd z_1\wedge \ldots \wedge\dd z_N$ is holomorphic in $W+ i\Gamma_{\delta_0}$. Hence, applying Stokes' theorem, one can rewrite the expression in \eqref{akeur}, using \eqref{loerihn}, as 
\begin{align}\label{lienurt}
\mathcal{F}_{p}\mu_{g,W}^{y}(\tau, \xi)&=  c_{p}\int_{ W} g\big(x+iy-i\delta \varphi(x)\tfrac{\xi}{|\xi|}\big)  e^{i(\tau-\theta(x)-iy)\xi-|\xi| p(\tau-\theta(x)-iy)} |\!\det \theta'(x)|\dd x.
\end{align}

We are now ready to conclude the proof. In fact,  estimating  \eqref{lienurt}, we have
\begin{align}\label{uyg6}
|\mathcal{F}_{p}\mu_{g,W}^{y}(\tau, \xi)|&\leq   c_{p}\int_{ W} \big| g\big(x+ iy-i\delta \varphi(x)\tfrac{\xi}{|\xi|}\big)\big|  e^{\Theta(x,\tau,\xi)} |\det \theta'(x)|\dd x,
\end{align}
where, in view of \eqref{;pocdr}, we can estimate
\begin{align}\label{inurhri}
\Theta(x,\tau,\xi)&:=\Re \{ i(\tau-\theta(x)-iy)\xi-|\xi| p(\tau-\theta(x)-iy)\}
\notag\\[4pt]
 &= y\cdot \xi-\delta \varphi(x)|\xi|-|\xi| \Re p(\tau- \theta(x)-i y)
 \notag\\[4pt]
& \leq - R|\xi|,  
\end{align}
for every $x \in W$, $\xi \in \Gamma'$ and $\tau\in V$.
Define the constant $A$ by
\begin{equation} \label{bjtfvvdvd}
A:=    c_{p} \Big(\sup_{\overline{W}+i K} |g(z)|\Big) \Big( \sup_{\RN}\{|\det \theta'(x)|\}\Big)m(W) < \infty.
\end{equation}
Note that $ \sup_{\RN}\{|\det \theta'(x)|\}$ is independent of $\tau$ since $\varphi$ is the composition of $\phi$ with the translation by $\tau$.
Collectively, from \eqref{uyg6}, \eqref{inurhri}, \eqref{bjtfvvdvd} and keeping in mind property \eqref{lfhr7464}, we obtain
\begin{align*}
|\mathcal{F}_{p}\mu_{g,W}^y(\tau, \xi)|&\leq  A e^{-R|\xi| }, \quad \forall (\tau,\xi) \in V\times\Gamma'
\end{align*}
as we wished to prove. \qedhere
\end{proof}

\begin{Cor}\label{WFSSBVH}
  Let $(x_0, \xi_0)\in \RN\times(\RN\setminus \{0\})$ and assume that there exist an open neighborhood $U$ of $x_0$ and finitely many open convex cones $\Gamma_j\subset \RN\setminus \{0\}$, $j= 1, \ldots, \ell$ that are opposite for $\xi_0$.
  Then, for every  open neighborhood $V\subset\subset U$ of $x_0$, there exists an open cone $\Gamma'\subset \RN$ with $\xi_0 \in \Gamma'$ such that for every family of holomorphic functions $F_j\in \mathcal{O}(\mathcal{W}_\delta(U; \Gamma_j)),$ $j\in\{1,\dots,\ell\}$, where $\delta>0$ is fixed, and any analytic functional $\nu$ that represents $\sum_{j=1}^{\ell}\bb F_j$ in $U$, there exist positive numbers $C$ and $R$ for which the inequality
  \begin{equation}\label{DesRA2}
    |\mathcal{F}_{p}\nu(\tau, \xi)|\leq C e^{-R |\xi|}, 
  \end{equation}
holds for every $\tau \in V$ and $\xi\in \Gamma'$.
\end{Cor}

\begin{proof}
Let $(x_0, \xi_0)$, $U,V$ and  $F_j \in \mathcal{O}(\mathcal{W}_\delta(U; \Gamma_j))$ for $j\in\{1,\dots,\ell\}$ as in the hypothesis.
From Proposition \ref{FBIBHBV} we know that, for each $j\in\{1,\dots,\ell\}$, there exist  conic neighborhoods $\Gamma_j'$ of $\xi_0$ such that if  $\nu_{F_j}$ is a representative of $\bb F_j$, then there are positive constants $R_j$ and $C_j$ for which the inequality 
  \begin{equation}\label{o7y}
    |\mathcal{F}_{p}\nu_{F_j}(\tau, \xi)|\leq C_j e^{-R_j|\xi|}, 
  \end{equation}
holds for every $\tau \in V$ and $\xi\in \Gamma'_j$.  
Define  $\Gamma':= \Gamma_1'\cap \cdots \cap \Gamma_\ell'$, $C:= \max_j C_j$ and $R:= \min_j R_j$. Thus, it follows from \eqref{o7y} that
  \begin{align*}
    |\mathcal{F}_{p}\nu(\tau, \xi)|&= \bigg|\sum_{j=1}^{\ell} \mathcal{F}_{p}\nu_{F_j}(\tau, \xi)\bigg|\leq \ell C e^{-R|\xi|},
  \end{align*}
for every $(\tau, \xi) \in V\times \Gamma',$ proving \eqref{DesRA2}.
\end{proof}

In particular, Corollary \ref{WFSSBVH} ensures that if $(x_0, \xi_0)\notin WF_{C^{\omega}}(u),$ then condition $(\mathfrak{R}_1)$ holds at $(x_0, \xi_0)$ proving one implication of Theorem~\ref{Theo2.4}. 

\subsection{Sharp analysis of the FBI inversion formula}

Let us take a closer look at the FBI inversion formula described in Proposition~\ref{fbibhif}. %In the arguments that follows, we want to describe a procedure that will be very usefull in the next section.

Let $V\subset\subset U$ be open neighborhoods of $x_0$ in $\RN$. Let $U\subset\subset W'$ be any open set in $\RN$ with smooth boundary and let $\chi$ be the characteristic function of  $W'$.
Given $u \in \mathcal{B}(\RN),$ let $\mu\in \mathcal{O}'(\overline{U})$ be a representative of $u$ in $U$. We consider the sequence given by the inversion formula \eqref{FBI:inversion}:
  \begin{align}\label{oieur}
    \mu_\epsilon(x)= \bigg( \int_{ \RN} \int_{\RN} e^{i(x- \tau)\xi -\epsilon |\xi|^{2}} \mathcal{F}_{p}\mu(\tau, \xi) |\xi|^{\frac{N}{2k}} \dd \tau \dd \xi\bigg) \frac{\chi (x)}{(2\pi)^{N}}
  \end{align}
and write $\mu_\epsilon(x):= \mu_0^{\epsilon}(x)+ \mu_1^{\epsilon}(x)$ where
  \begin{align}\label{oieurgbi}
    \mu_0^\epsilon(x):= \bigg( \int_{ \RN} \int_{|\tau- x_0|\leq a} e^{i(x- \tau)\xi -\epsilon |\xi|^{2}} \mathcal{F}_{p}\mu(\tau, \xi) |\xi|^{\frac{N}{2k}} \dd \tau \dd \xi\bigg) \frac{\chi (x)}{(2\pi)^{N}}
  \end{align}
and $a$ is a small real number such that $B_{a}(x_0) \subset V$.

\begin{Lem}\label{lemahdiur}
Let $\mathcal{S}$ be any subsheaf of $\mathcal B$ such that $C^\omega$ is a subsheaf of $\mathcal S$. If $\mu_0^{\epsilon} $ converges in $\mathcal{O}'(\overline{W'})$ to an analytic functional $\mu_0 \in \mathcal{O}'(\overline{W'})$, then  $(x_0, \xi_0)\in WF_{\mathcal{S}}(u)$ if and only if $(x_0, \xi_0) \in WF_{\mathcal{S}}(u_0)$, where $u_0\in \mathcal{B}(W')$ is represented by $\mu_0$.
\end{Lem}

\begin{proof}
Suppose that $\mu_0^{\epsilon} $ converges in $\mathcal{O}'(\overline{W'})$ to an analytic functional $\mu_0 \in \mathcal{O}'(\overline{W'})$, then $\mu_{1}^{\epsilon}= \mu^\epsilon- \mu_0^\epsilon$ converges to an analytic functional $\mu_1 \in \mathcal{O}'(\overline{W'})$ and we can consider the hyperfunctions $u_0, u_1 \in \mathcal{B}(W')$ defined respectively by their representatives $\mu_0$ and $\mu_1$.  Consider also $u^{\sharp}$ to be the hyperfunction in $W'$ represented by $\mu$. It is clear that $u^{\sharp}|_U= u|_U$ and that $u^{\sharp}= u_0+u_1$.
  The proof now follows from the next result.
  \end{proof}

\begin{Pro}\label{Claim1}
  Let $\epsilon_j$ be a sequence converging to $0$ such that $\mu_1^{\epsilon_j}$  converges to a holomorphic function on some ball $B_r^{\C}(x_0)$ for a small $r$, thus,  $u_1$ is a hyperfunction that is real-analytic in $B_{r}(x_0)$. 
%\qedhere
\end{Pro} 
\begin{proof}
See  Subsection~\ref{section5.4}.
\end{proof}

For the remainder of this subsection we are going to devote our attention to study $\mu_0$, that is,
%  
%  If we denote by $WF_\star(u)$ some flavor of wave-front set (in this section it stands for the analytic wave-front set), then our previous discussion implies that in a small neighborhood of $x_0$, $u_1$ is represented by an analytic function therefore  $(x_0, \xi_0)\in WF_{\star}(u)$ if and only if $(x_0, \xi_0) \in WF_\star(u_0)$.
%
we will show that the hypothesis made on Lemma~\ref{lemahdiur}, more precisely, that 
\begin{equation}
\mu_0^{\epsilon} \to\mu_0\ \text{ in }\ \mathcal{O}'(\overline{W'}), \quad \text{as } \ \epsilon \to 0,
\end{equation}
is always true if $W'$ is small enough.

Let us apply $\mu_{0}^{\epsilon}$ to an entire function $h\in \mathcal{O}(\CN)$ and perform two deformations in the contour of integration. First we apply $\xi \mapsto \xi+ i|\xi|(x-\tau)$ and denote the Jacobian determinant of this transformation by $\Delta(x-\tau, \xi)$,
\begin{align}
  \mu_0^{\epsilon}(h) =& \int_{W'} \frac{1}{(2 \pi)^{N}} \int_{\RN} \int_{|\tau- x_0|\leq a} e^{i (x-\tau)\xi- \epsilon|\xi|^{2}} \mathcal{F}_p \mu (\tau, \xi) |\xi|^{\frac{N}{2k}}  \dd \tau \dd \xi h(x) \dd x\nonumber\\
                      =& \int_{W'} \frac{1}{(2 \pi)^{N}} \int_{|\tau- x_0|\leq a} \int_{\RN}  e^{Q_\epsilon(x, \tau, \xi)}\mathcal{F}_p \mu (\tau,\zeta) \langle\zeta\rangle^{\frac{N}{2k}} \Delta(x-\tau, \xi) \dd \xi \dd \tau h(x) \dd x \nonumber
\end{align}
To perform the first deformation we need that $|\Re \zeta|> |\Im \zeta|$  so we will assume  that $W'= B_{1/4}(x_0)$ and that $a\leq 1/16$.
 The second deformation is $x \mapsto z= x+ i \chi(x)\xi/|\xi|$  where $\chi\in \Cinf_c(W')$  is such that $\chi\equiv 1/32$ in $B_{1/8}(x_0)$ and $0\leq \chi(x)\leq 1/32$.  We obtain:
\begin{align}
  \mu_0^{\epsilon}(h)
  =& \frac{1}{(2 \pi)^{N}} \int_{\RN} \int_{|\tau- x_0|\leq a} \int_{W'_{\ast}}  e^{Q_{\epsilon}(z,\tau, \xi)}\mathcal{F}_p \mu (\tau, \zeta') \langle\zeta'\rangle^{\frac{N}{2k}} \Delta(z-\tau, \xi)  h(z) \dd z \dd \tau \dd \xi \label{integrand}
\end{align}
where $W'_{\ast}= \{ x+i \chi(x) \xi/|\xi|: x \in W'\}$, 
\[
Q_\epsilon(z, \tau, \xi)= i (z-\tau)\xi- |\xi|\langle z-\tau\rangle^{2}- \epsilon\langle\xi + i|\xi|(z-\tau)\rangle^{2},
\]
$\zeta=  \xi + i|\xi|(x-\tau)$ and $\zeta'=  \xi + i|\xi|(z-\tau)$. 
Note that $|\Re \zeta'|> |\Im \zeta'|$.
Let us analyze the real part of $Q_\epsilon$:
\begin{align}\label{estQe}
\Re Q_\epsilon (z, \tau, \xi)&\leq  \Re\{  i (z-\tau)\xi- |\xi|\langle z-\tau\rangle^{2}\}\notag\\
	&= - \chi(x) |\xi| - |\xi|(|x-\tau|^{2}- |\chi(x)|^{2})\notag\\	
	&\leq \left\{\begin{array}{lr}
            -\frac{31}{32^{2}}|\xi|& \quad \quad \textrm{ if } x \in B_{1/8}(x_0), \tau \in B_{a}(x_0)\\[5pt]
            -\frac{3}{32^{2}}|\xi|& \quad \quad \textrm{ if } x \notin B_{1/8}(x_0), \tau \in B_{a}(x_0).
  \end{array}\right.
\end{align}
Inequality \eqref{estQe} together \eqref{SEMPREVALE} shows that the integrand in \eqref{integrand} is bounded independently of $\epsilon$ by an integrable function and one can apply the dominated convergence theorem to obtain
\begin{align}\label{integralrepresentation}
  \mu_0(h)  = \frac{1}{(2 \pi)^{N}} \int_{\RN} \int_{|\tau- x_0|\leq a} \int_{W'_{\ast}}  e^{i (z-\tau)\xi- |\xi|\langle z-\tau\rangle^{2}}\mathcal{F}_p \mu (\tau, \zeta') \langle\zeta'\rangle^{\frac{N}{2k}} \Delta(z-\tau, \xi)  h(z) \dd z \dd \tau \dd \xi.
\end{align}

Let us now prove that $\mu_0\in \mathcal{O}'(\overline{W'})$ by proving that $\mu_0$ satisfies an estimate of the kind given in \eqref{deffuncanal}. 
To see this, fix a positive number $\delta$. If $\delta>1/32$ then it is clear from the integral representation of $\mu_0$, equation \eqref{integralrepresentation}, that there is $C>0$ such that
\begin{align*}
  |\mu_0(h)|\leq C \sup_{z\in  W'_{\delta}}|h(z)|.
\end{align*}

Let us treat the case where $\delta\in (0,1/32]$. We note that our construction of the integral representation of $\mu_0$ integrates the $z$-variable in a  deformation of the real open set $W'$ inside the complex open set $W'_{1/32}$ and in order to integrate the $z$-variable in a deformation of $W'$ inside of $W'_{\delta}$  we consider a decomposition  $B_{a}(x_0)= \bigcup_{j\in I} A_j$ by a finite union of measurable disjoint sets of radius at most $\delta a$, i.e., for each $j\in I$ there is $x_j$ such that $A_j \subset B_{\delta a}(x_j)$. Then
\begin{align*}
  \mu_0(h)  &= \frac{1}{(2 \pi)^{N}} \int_{\RN} \int_{|\tau- x_0|\leq a} \int_{W'_{\ast}}  e^{i (z-\tau)\xi- |\xi|\langle z-\tau\rangle^{2}}\mathcal{F}_p \mu (\tau, \zeta') \langle\zeta'\rangle^{\frac{N}{2k}} \Delta(z-\tau, \xi)  h(z) \dd z \dd \tau \dd \xi\\
  &=   \sum_{j \in I}\lim_{\epsilon \lra 0} \int_{W'} \frac{1}{(2 \pi)^{N}} \int_{A_j} \int_{\RN}  e^{Q_\epsilon(x, \tau, \xi)}\mathcal{F}_p \mu (\tau,\zeta) \langle\zeta\rangle^{\frac{N}{2k}} \Delta(x-\tau, \xi) \dd \xi \dd \tau h(x) \dd x.
\end{align*}
Now let $\chi_j\in \Cinf_c(W')$ such that $\chi_j= \delta/32$ in $B_{\delta/8}(x_j)$, $\chi_j=0$ in $W'\setminus B_{\delta/4}(x_j)$ and $0\leq \chi_j\leq \delta/32$. Given $\xi \in \RN$ consider the deformation $x\mapsto z_j= x+i \chi_j(x)\xi/|\xi|$ and denote $W'_{j,\ast}= \{ x+i\chi_j(x)\xi/|\xi|: x \in W'\}.$ Then
\begin{align*}
  \Re\{  i (z_j-\tau)\xi- |\xi|\langle z_j-\tau\rangle^{2}\}&= - \chi_j(x) |\xi| - |\xi|(|x-\tau|^{2}- |\chi_j(x)|^{2})\\
  &\leq \left\{\begin{array}{lr}
            -\frac{32\delta-\delta^{2}}{32^{2}}|\xi|& \quad \quad \textrm{ if } x \in B_{\delta/8}(x_0), \tau \in B_{\delta a}(x_0)\\[5pt]
            -\frac{3 }{32^2}\delta^{2}|\xi|& \quad \quad \textrm{ if } x \notin B_{\delta/8}(x_0), \tau \in B_{\delta a}(x_0)
  \end{array}\right.
\end{align*}
If we denote by $\zeta'_j= \xi+ i|\xi|(z_j-\tau)$ for every $j\in I$, then we can write
\begin{align}\label{integralrepresentation2}
  \mu_0(h) &=   \sum_{j \in I} \frac{1}{(2 \pi)^{N}} \int_{A_j} \int_{\RN}  \int_{W'_{j,\ast}} e^{i (z-\tau)\xi- |\xi|\langle z-\tau\rangle^{2}}\mathcal{F}_p \mu (\tau,\zeta_j') \langle\zeta_j'\rangle^{\frac{N}{2k}} \Delta(x-\tau, \xi) h(z)\dd z\dd \xi \dd \tau.
\end{align}
Therefore, we have
\begin{align*}
  |\mu_0(h)|\leq C_\delta \sup_{z \in W'_{\delta}} |h(z)|.
\end{align*}
Proving that $\mu_0$ is carried by $\overline{W'}$, as we wished to prove.

\subsection{Microlocal decomposition of an analytic functional}

Now we know that $\mu_0^{\epsilon}$ converges in $\mathcal{O}'(\overline{W'})$ to $\mu_0$, thus, $\mu_1^{\epsilon}= \mu_\epsilon- \mu_0^{\epsilon}$ also converges to $\mu_1$ in $\mathcal{O}'(\overline{W'})$. Proposition~\ref{Claim1} states that $\mu_1$ is real-analytic in a neighborhood of $x_0$, therefore, $\mu$ and $\mu_0$ have the same microlocal regularity at $x_0$ (with respect to any sheaf that contains $C^{\omega}$). Our next procedure will show how to manufacture a decomposition of $\mu_0$ as sum of analytic functionals in a sharp way in which only one of them encodes the  microlocal regularity at $(x_0, \xi_0)$, i.e., the other analytic functionals define hyperfunctions that are microlocal real-analytic at $(x_0, \xi_0)$.

Now we choose $\mathcal{C}_j$, $1\leq j\leq L$ open, acute cones such that
  \begin{align}\label{conesC1L}
    \RN= \bigcup_{j=1}^{L} \overline{\mathcal{C}_j}, \quad \text{and}\quad {\rm int}(\mathcal{C}_j\cap \mathcal{C}_k)=\emptyset, \ \forall\, j,k\in\{1,\cdots,L\}, j\ne k.
  \end{align}
We can further assume that $\xi_0 \in \mathcal{C}_1\subset \Gamma$ and $\xi_0 \notin \overline{\mathcal{C}_j}$ when $j\geq 2$. 
Let $\Gamma_j$, $j \in \{1,\ldots, L\}$  be open cones with the property that, for some $c$, $0< c\leq 1$, it holds
  \begin{align}\label{liaus}
      v \cdot \xi\geq c |v||\xi|,\quad  \forall v \in \Gamma_j, \forall \xi \in \mathcal{C}_j,
  \end{align}
and $\xi_0 \cdot \Gamma_j <0$ for $j\in \{2, \ldots, L\}.$ 
The later implies that $\Gamma_j$ is opposite to $\xi_0$ for $j\in \{2, \ldots, L\}.$
  
For each $j\in\{1,\dots,L\},$ 
\begin{align}\label{partedasefes}
  f_j(z) = \frac{1}{(2\pi)^{N}}\int_{|\tau- x_0|\leq a} \int_{\mathcal{C}_j}  e^{i\xi(z-\tau)-|\xi|\langle z-\tau\rangle^{2}} \mathcal{F}_{p} \mu(\tau, \zeta') \langle\zeta'\rangle^{\frac{N}{2k}} \Delta(z-\tau, \xi)  \dd \xi \dd \tau
\end{align}
defines a holomorphic function in $\RN+ i\Gamma_j^{c}$, where $\Gamma_j^{c}= \{ y\in \Gamma_j: |y|< c\},$ recall that $\zeta'= \xi+ i|\xi|(z-\tau)$.

We want to prove that $\mu_0- \sum_{j=1}^{L} \mu_{f_j,W'} \in \mathcal{O}'(\del W')$. To this end, let $\mathcal{U}$ be a complex neighborhood of $\del W'$. By Remark~\ref{faluh}, item (i), we can find $\epsilon>0$ such that if $y_j \in \Gamma_j^{c}$ is such that $|y_j|\leq \epsilon$, then $\mu_{f_j, W'}- \mu_{f_j, W'}^{y_j} \in \mathcal{O}'(\mathcal{U})$. Therefore, it is enough to study $\mu_0 - \sum_{j=1}^{L}\mu_{f_j, W'}^{y_j}$. To do so, let $h \in \mathcal{O}(\CN)$ we denote
\begin{align*}
  \Theta^{(\xi, \tau)}(z)= e^{i(z-\tau)\xi -|\xi|\langle z-\tau\rangle^2} \mathcal{F}_{p} \mu(\tau, \zeta') \langle\zeta'\rangle^{\frac{N}{2k}} \Delta(z-\tau, \xi)   h(z)\frac{ \dd z}{(2\pi)^{N}},
\end{align*}
where $\Theta$ is a holomorphic $N$-form.

To obtain the decomposition of $\mu_0$ we will use the integral representations of $\mu_0$ given in \eqref{integralrepresentation} and \eqref{integralrepresentation2}. Let us assume, for simplicity, that $y_j+ B_{1/32}(0)\subset \Gamma_j^c$ for every $j=1, \ldots, L$ and argue using the integral representation of $\mu_0$ given in \eqref{integralrepresentation}. Otherwise, we find $\delta\in (0,1)$ such that $y_j+ B_{\delta/32}(0)\subset \Gamma_j^c$ for every $j=1, \ldots, L$ and repeat our next argument using \eqref{integralrepresentation2} instead of \eqref{integralrepresentation}.

Let $W'_{\ast}= \{ x+i \chi(x) \xi/|\xi|: x \in W'\}$ be as in \eqref{integralrepresentation}, assuming that $y_j+ B_{1/32}(0)\subset \Gamma_j^c$ for every $j=1, \ldots, L$ we have another representation for the action of  $\mu_{f_j, W'}^{y_j}$ in $h \in \mathcal{O}(\CN)$:
\begin{align*}
  \mu_{f_j, W'}^{y_j}(h)=&   \int_{W'} f_j(x+iy_j) h(x+iy_j) \dd x\\
  =&  \int_{\mathcal{C}_j} \int_{|\tau- x_0|\leq a} \int_{W'_{\ast}} \Theta^{(\xi, \tau)}(z+iy_j) \dd \tau \dd \xi. 
\end{align*}

 For every $j \in \{1, \ldots, L\}$, consider
\begin{align*}
\mathcal{W}_j= \{\tilde{z} \in \CN: \exists t \in [0,1]: \tilde{z} - t y_j \in W'_{\ast}\}.  
\end{align*}
Applying the Stokes' theorem for $\dd_z \Theta^{(\xi, \tau)}$ in $\mathcal{W}_j$ we obtain that
\begin{align*}
   \int_{W'_{\ast}}\big(\Theta^{(\xi, \tau)}(z)-\Theta^{(\xi, \tau)}(z+iy_j) \big)=  \int_{0}^{1}\int_{\del W'} \Theta^{(\xi, \tau)}(z_{j,t}) 
\end{align*}
where $z_{j,t}= z+it y_j$.
Thus we can write $\mu_0 - \sum_{j=1}^{L} \mu_{f_j,W'}^{y_j}= \sum_{j=1}^{L} \lambda_j$ where
\begin{align*}
  \lambda_j(h)=\int_{\mathcal{C}_j} \int_{|\tau- x_0|\leq a} \int_{0}^{1}\int_{\del W'} \Theta^{(\xi, \tau)}(z_{j,t}) \dd \tau \dd \xi.
\end{align*}
Therefore, $\lambda_j \in \mathcal{O}'(\del W'+ i[0,1]y_j)$. Hence $\mu_0 - \sum_{j=1}^{L} \mu_{f_j, W'}^{y_j} \in \mathcal{O}'(\mathcal{U}),$ and this implies that $\mu_0 - \sum_{j=1}^{L} \mu_{f_j, W'}\in \mathcal{O}'(\del W').$ 

Now define
\begin{align}\label{Fum}
  F_1(z)= \frac{1}{(2\pi)^{N}} \int_{|\tau-x_0|\leq  a} \int_{\mathcal{C}_1} e^{i\xi(z-\tau)} \mathcal{F}_p \mu (\tau, \xi) |\xi|^{\frac{N}{2k}} \dd \xi \dd \tau
\end{align}
and
\begin{align*}
  R_{1}(z)= \frac{1}{(2\pi)^{N}} \int_{|\tau-x_0|\leq  a}\int_{0}^{1} \int_{\del \mathcal{C}_1} e^{i\zeta_t(z- \tau)} \mathcal{F}_p \mu (\tau, \zeta_t) \langle\zeta_t\rangle^{\frac{N}{2k}} \dd \zeta_t\dd \tau
\end{align*}
where $\zeta_t= \xi+ i t|\xi|(z- \tau)$. Both $F_1$ and $R_1$ are holomorphic functions in $\RN+i\Gamma_1^{c}$. Note that
\begin{align*}
  \mu_{f_1, W'}- \mu_{F_1, W'}- \mu_{R_1, W'}=& (\mu_{f_1, W'}- \mu_{f_1, W'}^{y_1})- (\mu_{F_1, W'}-\mu_{F_1, W'}^{y_1})- (\mu_{R_1, W'}-\mu_{R_1, W'}^{y_1})\\
  &+   \mu_{f_1, W'}^{y_1}- \mu_{F_1, W'}^{y_1}- \mu_{R_1, W'}^{y_1}.
\end{align*}
Given $\mathcal{U}$ a complex open neighborhood of $\del W'$ there is $\epsilon'>0$ such that if we choose $y_1\in \Gamma^{c}_{1}$ with $|y_1|< \epsilon'$, then analytic functionals  $\mu_{f_1, W'}- \mu_{f_1, W'}^{y_1}$, $\mu_{F_1, W'}-\mu_{F_1, W'}^{y_1}$ and $\mu_{R_1, W'}-\mu_{R_1, W'}^{y_1}$ are carried by $\mathcal{U}$ and $\mu_{f_1, W'}^{y_1}- \mu_{F_1, W'}^{y_1}- \mu_{R_1, W'}^{y_1}=0$ thanks to Stokes' theorem.

We can write $\del \mathcal{C}_1= \bigcup_{j=2}^{L} D_j$ where each $D_j$ is a measurable subset of $\del \mathcal{C}_1$ such that $D_j\cap D_k = \emptyset$ if $j\neq k$ and $D_j \subset \mathcal{C}_j$. Chosen $\tilde{v_1} \in \Gamma_1$ and $\tilde{v_j} \in \Gamma_j$ we have that
\begin{align*}
  (t\tilde{v_1}+ (1-t)\tilde{v_j})\xi= t\tilde{v_1}\xi+ (1-t)\tilde{v_j}\xi\geq c \big(t|\tilde{v_1}|+ (1-t)|\tilde{v_j}|)\xi, \quad \forall\xi\in D_j.
\end{align*}
Let $\tilde{\Gamma}_j$ be the convex hull of $\Gamma_1\cup\Gamma_j$, so $v\cdot \xi\geq c|v||\xi|$ for every $v\in \tilde{\Gamma}_j$ and $\xi\in D_j$, this means that if we define
\begin{align*}
  R_{j}(z)= \frac{1}{(2\pi)^{N}} \int_{|\tau-x_0|\leq  a}\int_{0}^{1} \int_{D_j} e^{i\zeta_t(z- \tau)} \mathcal{F}_p \mu (\tau, \zeta_t) \langle\zeta_t\rangle^{\frac{N}{2k}} \dd \zeta_t\dd \tau,
\end{align*}
for every $j \in \{2, \ldots, L\}$, then $R_j$ is a holomorphic function in $\tilde{\Gamma}_j$ and $R_1= \sum_{j=2}^{L} R_j$. Therefore, $\mu_{R_1,W'}= \sum_{j=2}^{L} \mu_{R_j,W'}$, i.e., $\mu_{R_1, W'}$ is the boundary value of holomorphic functions defined in cones opposing the direction $\xi_0$.

We conclude that $\mu_0= \mu_{F_1, W'}+ \sum_{j=2}^{L}\big(\mu_{f_j, W'}+ \mu_{R_j,W'}\big)$ where $\mu_{F_1, W'}$ is the only term that is not the boundary value of a holomorphic function defined in a cone opposing the direction $\xi_0$.

\begin{Def}\label{MLdecom}
  We shall say that decomposition described above
  \begin{align}
\mu_0=  \mu_{F_1, W'}+ \sum_{j=2}^{L}\big(\mu_{f_j, W'}+ \mu_{R_j,W'}\big)    
  \end{align}
 is a \it{microlocal decomposition for $\mu_0$}. 
\end{Def}
Such decomposition is important in our microlocal analysis because allow us to study separately the ``part'' of $\mu_0$ that concentrate the information in the direction $\xi_0$   from the ``part''  defined by holomorphic functions in cones  opposing the direction $\xi_0$. 

Thanks to this decomposition we have the following refinement of 
Lemma~\ref{lemahdiur}.
\begin{Pro}\label{FundamentalRemark}
Let $\mu\in \mathcal{O}'(\overline{U})$ a representative of $u$. Define $F_{1}$ as in \eqref{Fum}. If there is a neighborhood $W''$ of $x_0$ such that the restriction of $\bb F_1$ to $W''$ belongs to $\mathcal{S}(W'')$, then $(x_0,\xi_0)\notin WF_{\mathcal{S}}(u)$.
\end{Pro}

We are going to prove  Theorem~\ref{Theo2.4} using the above proposition with $\mathcal{S}= C^{\omega}$.

\begin{Pro}\label{laiw} Let $u \in \mathcal{B}(\RN)$. Assume that there are open neighborhoods $U, V$ of $x_0$ in $\RN$ with $V\subset\subset U$,   $\mu \in \mathcal{O}'(\overline{U})$ representative of $u$ in $U$, a conic neighborhood $\Gamma$ of $\xi_0$ in $\RN$ and constants $c_1, c_2>0$ such that
  \begin{align}\label{aleuyf}
    |\mathcal{F}_{p}\mu(\tau, \xi)|\leq c_1 e^{-c_2|\xi|}, \quad \forall (\tau, \xi) \in V\times \Gamma.
  \end{align}
Then $F_1$, defined in \eqref{Fum}, is a holomorphic function in a full neighborhood of $x_0$. And consequently $(x_0, \xi_0) \notin WF_{C^\omega}(u).$
\end{Pro}

\begin{proof}
The exponential decay of $\mathcal{F}_{p}\mu(\tau, \xi)$ on $ \{ |\tau- x_0|\leq a\}\times \mathcal{C}_1$, \eqref{aleuyf},  guarantees that $F_1$ is a holomorphic function in $\RN+ i \{ y \in \Rn: |y|< c_2/2\}$.
\end{proof}

% \blue{
% \begin{Rem}
% One can think Theorem~\ref{Theo2.4}  as a self-improving theorem and give a new definition of analytic wave front set that will be very much useful in other classes of functions. In one direction, we start with
% a hyperfunction $u$ and a covector $(x_0,\xi_0)\notin WF_a(u)$. Applying one direction of Theorem~\ref{Theo2.4} we obtain \eqref{aleoru}. Now, one can follow the proof of the other direction of Theorem~\ref{Theo2.4} to  conclude that there exist cones $\Gamma_2, \ldots, \Gamma_L$ and  $f_j\in\mathcal{O}(\mathcal{W}_\delta(W, \Gamma_j))$ for each $j=2, \ldots, L$ such that
%   $u- \bb f_j$ is a real-analytic function.
% \end{Rem}}

\subsection{Proof of Proposition~\ref{Claim1}}\label{section5.4}

We asserted in Proposition~\ref{Claim1} that the restriction of the hyperfunction defined by $\mu_1$ to a small neighborhood of $x_0$ is real-analytic. To prove this fact, consider the entire function
\begin{align}\label{ALSOENTIRE}
  z &\mapsto \frac{1}{(2\pi)^{N}}\int_{\RN} \int_{|\tau-x_0|\geq a} e^{i(z- \tau)\xi- \epsilon|\xi|^{2}} \mathcal{F}_{p}\mu(\tau, \xi)|\xi|^{\lambda N} \dd \tau \dd \xi.
\end{align}

 Therefore, we can regard $\mu_1^{\epsilon}$ as a sequence of holomorphic functions defined in $W+i\RN$. 
\begin{Pro}\label{ConvergenceLemma}
  There exist $r>0,$ a sequence $\{\epsilon_k\}_{k\in\N}$ converging to zero and a holomorphic function $H \in \mathcal{O}(B_r^\C(x_0))$ such that $\mu_1^{\epsilon_k}$ converges to $H$ in $\mathcal{O}( B_r^{\C}(x_0))$.
\end{Pro}
\begin{proof}
First, we point out that we are going to choose $r$ small such that $B_r(x_0)\subset W,$ so we can disregard the term $\chi/(2\pi)^N$ from the definition of $\mu_1^\epsilon$ and we can consider $\mu_1^{\epsilon}$ as a holomorphic function in $B_r^\C(x_0)$. We can write:
  \begin{align}\label{oieurgbi2}
    \mu_1^\epsilon(z)&=  \int_{ \RN} \int_{|\tau- x_0|\ge a} e^{i(z- \tau)\xi -\epsilon |\xi|^{2}} \mathcal{F}_{p}\mu(\tau, \xi) |\xi|^{\frac{N}{2k}} \dd \tau \dd \xi
\notag\\[5pt]
&=\mu_2^{\epsilon}(z)+ \mu_{3}^{\epsilon}(z) + \mu_4^{\epsilon}(z)
  \end{align}
%
%We will write
%  \begin{align*}
%    \mu_1^{\epsilon}(z) = \mu_2^{\epsilon}(z)+ \mu_{3}^{\epsilon}(z) + \mu_4^{\epsilon}(z),
%  \end{align*}
 where, for a  positive number $A>a$ to be determined, 
    \begin{align*}
    \mu_2^{\epsilon} &\textrm{ is the integral over } X_2=\{ (\tau, \xi): a\leq |\tau-x_0|\leq A, |\xi|\leq 1\}; \\
    \mu_3^{\epsilon} &\textrm{ is the integral over } X_3=\{ (\tau, \xi): |\tau-x_0|\geq A, |\xi|\in \RN\}; \\
    \mu_4^{\epsilon} &\textrm{ is the integral over } X_4=\{ (\tau, \xi): a\leq |\tau-x_0|\leq A, |\xi|\geq 1\}. 
  \end{align*}

It will be enough to prove that there exist $r>0$ and $M>0$ such that $|\mu_j^{\epsilon}(z)|\leq M$ for all $z \in B_r^\C(x_0)$, every $0< \epsilon< 1$ and $j\in\{1,2,3\}$. Since $X_2$ is a compact set it is clear that if $z$ is in a bounded neighborhood of $x_0$, then $\mu_2^{\epsilon}(z)$ is bounded independently of $\epsilon$.

For $\mu_3^{\epsilon}(z)$ our goal is to guarantee that $X_3$ is chosen such that $|\Re p(\tau-w)|\geq |\tau-w|^{2k}$. It is here that we will choose $A$. First, let $l\geq 1$ be such that $\overline{U} \subset B_{l}(x_0)$ and define $A= 2l$, thus
\begin{equation}\label{faiu}
|\tau-x_0|\geq A\  \text{ and }  |w-x_0|\leq l \ \text{ implies }\ |\tau- w|\geq |\tau-x_0|/2\geq A/2\ge1.
\end{equation}
Let $\delta>0$ be small enough so that $\overline{U}_{\delta}\subset B_{l}^\C(x_0)$ and also 
\begin{equation}\label{o87bt}
\tau-w \in \Gamma_\rho \ \text{ for all }\ w\in \overline{U}_{\delta} \ \text{ and } \ \tau\in \RN \ \text{ satisfying }\ |\tau-x_0|\geq A, 
\end{equation}
for some $\rho>0$ for which the inequality \eqref{DesParComp} holds.

With these choices, one can estimate $\big|\mu_3^{\epsilon}(z)\big|$ as follows
  \begin{align*}
    \big|\mu_3^{\epsilon}(z)\big|       &= c_p \Big|\mu_w\Big(\int_{\RN} \int_{|\tau-x_0|\geq A} e^{i(z-w)\xi- \epsilon|\xi|^{2} -|\xi| p (\tau-w)}|\xi|^{\frac{N}{2k}} \dd \tau \dd \xi\Big)\Big|
\\[5pt]
                        &\leq c_pC_\delta\sup_{w \in \overline{U}_{\delta}} \int_{\RN} \int_{|\tau-x_0|\geq A} e^{\Re \{i(z-w)\xi\} -|\xi| c'|\tau-w|^{2k}}|\xi|^{\frac{N}{2k}} \dd \tau \dd \xi
\\[5pt]
&\leq c_pC_\delta\sup_{w \in \overline{U}_{\delta}} \int_{\RN} \int_{|\tau-x_0|\geq A} e^{\Re \{i(z-w)\xi\}}e^{-\frac{|\xi| c' A^2}{8}}e^{ -|\xi| \frac{c'}{2}|\tau-w|^{2}}|\xi|^{\frac{N}{2k}} \dd \tau \dd \xi\\
            &\leq C' \sup_{w \in \overline{U}_{\delta}} \int_{\RN}  e^{\Re \{i(z-w)\xi\} }e^{-\frac{ c' A^2}{8}|\xi|}|\xi|^{- \frac{(k-1)N}{2k}} \dd \xi.
  \end{align*}
%
%\begin{equation}\label{ihurgig}
%\parbox{12.5cm}{
Choosing $r$ small enough so that  $|\Im (z-w)|< c' A^2/8$ for every $z\in B_r(x_0)$ we see that there exists a positive constant $C''$ independent of $\epsilon$ such that $|\mu^{\epsilon}_3(z)|\leq C''$.
%}
%\end{equation}

Now we consider $\mu_4^\epsilon$:
  \begin{align*}
    |\mu_4^{\epsilon}(z)|  &=c_p\Big|\mu_w\Big( \int_{a\leq |\tau-x_0|\leq A} \int_{|\xi|\geq 1}e^{i(z-w)\xi- \epsilon|\xi|^{2} -|\xi| p (\tau-w)}|\xi|^{\frac{N}{2k}}  \dd \xi \dd \tau\Big)\Big|\\
    &\leq C_{\delta} \sup_{w\in \overline{U}_\delta} \bigg| \int_{a\leq |\tau-x_0|\leq A} \int_{|\xi|\geq 1}e^{i(z-w)\xi- \epsilon|\xi|^{2} -|\xi| p (\tau-w)}|\xi|^{\frac{N}{2k}}  \dd \xi \dd \tau \bigg|.
  \end{align*}
  The last inequality ensures that to estimate $|\mu_4^{\epsilon}(z)|$ with $z\in B_r^\C(x_0)$, it is enough to find a positive constant $M$ independent of $\epsilon$ such that
\begin{equation}\label{faliue}
\displaystyle q(z, w, \epsilon,\tau)=\Big| \int_{|\xi|\geq 1}e^{i(z-w)\xi- \epsilon|\xi|^{2} -|\xi| p (\tau-w)}|\xi|^{\frac{N}{2k}}  \dd \xi\Big|\leq M
\end{equation}
for every $\epsilon>0$, $z\in B_r(x_0)$, $w \in \overline{U}_{\delta}$ and $\tau\in\RN$ satisfying $a\leq |\tau-x_0|\leq A.$
Since $q$ is a continuous function it will be enough to prove \eqref{faliue} when $\Re(z-w) \neq 0$.

 Choose $s>0$ such that $s|\Re(z-w)|<1/2$ for every $z\in B_r^\C(x_0)$ and every $w\in \overline{U}_\delta.$ Fix $z\in B_r^\C(x_0)$, $w\in \overline{U}_\delta$ with $\Re(z-w) \neq 0$. Now fix $R>1$ and consider 
  \begin{align*}
  \Omega_R= \{ \zeta \in \CN: 1<|\Re\zeta|<R, \Im \zeta= \sigma s \Re(z-w)|\Re \zeta|\textrm{ for some } \sigma\in(0,1) \}.
  \end{align*}
Note that $\Omega_R\subset \Gamma_{\frac{1}{2}}\dot{=} \{\zeta\in \CN: |\Im \zeta|< \frac{1}{2}|\Re \zeta|\}$ thus
  \begin{equation}\label{laueyrt}
 \frac{\sqrt{3}}{2}|\Re \zeta|\leq \Re \langle \zeta\rangle\leq |\Re \zeta|\quad \text{and}\quad | \Im \langle \zeta\rangle|\leq \frac{1}{2}|\Re \zeta|, \quad \forall\, \zeta \in \Gamma_{\frac{1}{2}}.
  \end{equation}
The boundary in the Stokes' theorem sense of $\Omega_R$ is given by the disjoint union 
\[
(\del \Omega_R)_1\cup (\del \Omega)_2\cup (\del \Omega_R)_3\cup(\del \Omega_R)_4
\]
where:
  \begin{align*}
    (\del \Omega_R)_1&=\{ \zeta \in \CN: 1< |\Re\zeta| < R \textrm{ and } \Im \zeta=0 \};\\
    (\del \Omega)_2 &=\{ \zeta \in \CN: |\Re\zeta|=1 \textrm{ and } \Im \zeta= \sigma s \Re(z-w)\textrm{ for some } \sigma\in[0,1] \};\\
    (\del \Omega_R)_3 &=\{ \zeta \in \CN: 1< |\Re\zeta|< R \textrm{ and } \Im \zeta=  s \Re(z-w)|\Re \zeta| \};\\
    (\del \Omega_R)_4 &=\{ \zeta \in \CN: |\Re\zeta|=R \textrm{ and } \Im \zeta= \sigma s \Re(z-w)R\textrm{ for some } \sigma\in[0,1] \}.
  \end{align*}
Applying the Stokes' theorem in $\Omega_R$ to the closed holomorphic $N$-form
  \begin{align}\label{oe8rn}
    \omega_\epsilon(\zeta)= e^{Q(z,w,\zeta,\epsilon)}\langle\zeta\rangle^{\frac{N}{2k}}  \dd \zeta,
  \end{align}
where
\begin{align}\label{Qagain}
   Q(z,w,\zeta,\epsilon):=i(z-w)\zeta- \epsilon\langle \zeta\rangle^{2} -\langle \zeta\rangle p (\tau-w)
\end{align}
we obtain
  \begin{align}\label{kanie}
    \int_{1\leq |\xi|\leq R}e^{i(z-w)\xi- \epsilon|\xi|^{2} -|\xi| p (\tau-w)}|\xi|^{\frac{N}{2k}}  \dd \xi= - \int_{(\del \Omega)_2} \omega_\epsilon- \int_{(\del \Omega_R)_3} \omega_\epsilon - \int_{(\del \Omega_R)_4} \omega_\epsilon.
  \end{align}

We will now analyze each one of the integrals in the right hand-side of \eqref{kanie} separately:

\medskip

\noindent\underline{The integral over $(\del \Omega)_2$:}
Clearly  $\omega_\epsilon$ can be bounded independently of $\epsilon$ on $(\del \Omega)_2$ and it follows that $\int_{(\del\Omega)_2}\omega_\epsilon$ is uniformly bounded.

\medskip

\noindent\underline{The integral over $(\del \Omega_R)_4$:}
We claim that 
\begin{equation}\label{iurhas}
\int_{(\del \Omega_R)_4} \omega_\epsilon\lra 0 \quad \text{ when } \quad R\lra +\infty.
\end{equation}
In fact, in this case we can estimate $\Re Q(z,w,\zeta,\epsilon)$ with $Q$ given in  \eqref{Qagain} as
  \begin{align*}
    \Re Q(z,w,\zeta,\epsilon)&=\Re\{ i(z-w)\zeta- \epsilon\langle \zeta\rangle^{2}- \langle \zeta\rangle p(\tau- w)\}
\\[5pt]
& \leq  R \big[-s\sigma |\Re (z-w)|^{2} + |\Im (z-w)| + |\Re p(\tau-w)|+ \frac{1}{2} |\Im p(\tau-w)|
\\[5pt]
&\hspace{10pt}- \epsilon R_0 (1-s^{2}|\Re( z- w)|^{2})\big], 
    \end{align*}
for any $ 0<R_0\leq R$.    Since $z$, $\tau$ and $w$ are varying in bounded sets  we can further estimate $\Re Q(z,w,\zeta,\epsilon)$ as
  \begin{align}\label{fawiue}
    \Re Q(z,w,\zeta,\epsilon)& \leq  R \big[C - \epsilon R_0 (1-s^{2}|\Re( z- w)|^{2})\big].
    \end{align}
If $R$ is big enough, we can find $R_0$ such that the expression between brackets in \eqref{fawiue} is negative.
Fixing such $R_0$ and using the fact that $\langle \zeta\rangle^{N/2k}$ is homogeneous of degree $N/2k$ we conclude the claim.  

Therefore, for each $\epsilon>0$ fixed we are allowed to make $R\to\infty$ in \eqref{kanie}, to obtain
  \begin{align*}
    \int_{1\leq |\xi|}e^{i(z-w)\xi- \epsilon|\xi|^{2} -|\xi| p (\tau-w)}|\xi|^{\frac{N}{2k}}  \dd \xi= - \int_{(\del \Omega)_2} \omega_\epsilon- \int_{(\del \Omega)_3^{\ast}} \omega_\epsilon
  \end{align*}
  where $(\del \Omega)_3^{\ast}=\{ \zeta \in \CN: 1\leq |\Re\zeta|\textrm{ and } \Im \zeta=  s \Re(z-w)|\Re \zeta| \}$.

\medskip

\noindent\underline{The integral over $(\del \Omega)_3^\ast$:}
  Now it remains to  prove  that  $\int_{(\del\Omega)_3^{\ast}}\omega_\epsilon$ converges independently of $\epsilon$.
In fact, choose $\lambda$ small and $r_0>0$ such that $a-r_0>2\lambda$. Hence
\begin{equation}\label{aksueyt}
\text{ if $a\leq |\tau- x_0|$ and $|z-x_0|< r_0$, then $|z-\tau|> 2\lambda$.}
\end{equation} 
We are going to assume that $0<r\leq r_0$. For a technical reason  that will be clarified soon consider $f:[0,1] \lra \R$ given by
  \begin{align}\label{eqf1}
    f(t)= p(\tau- \Re w- i t\Im w).
  \end{align}
The mean value theorem guarantees that there exists $t^{\ast} \in [0,1]$ such that
  \begin{align}\label{eqf12}
    p(\tau- w)- p(\tau- \Re w)= f(1)-f(0)= f'(t^{\ast}).
  \end{align}
  Note that
  \begin{align*}
    f'(t)= \sum_{j=1}^N \frac{\del p}{\del z_j} (\tau - \Re w- i t \Im w)\cdot (-i\Im w_j).
  \end{align*}
  Therefore, we use again that $\tau$ and $w$ are in compact sets together with $|\Im w|\leq \delta$ to obtain that there exists a positive constant $B$ in which
\begin{equation}\label{qe4bnr0}
|f'(t^{\ast})|\leq B \delta.
\end{equation}
Using \eqref{eqf1} and  \eqref{eqf12} we can rewrite $\Re Q(z,w,\zeta,\epsilon)$ with $Q$ given by  \eqref{Qagain} as:
  \begin{align}\label{semponsro}
    \Re Q(z,w,\zeta,\epsilon)%&=
%\\[5pt]
=&- s|\xi||\Re z-\Re w|^2 - (\Im z- \Im w)\xi - \Re \langle \zeta \rangle  p(\tau-\Re w)
\\[5pt]
& -\Re \langle \zeta \rangle  \Re f'(t^{\ast})+\Im \langle \zeta \rangle  \Im f'(t^{\ast})- \epsilon (|\xi|^{2} -s^2|\xi|^{2}|\Re z-\Re w)|^{2}).
\notag
  \end{align}
Our choice of $s$ implies that $|\xi|^{2} -s^2|\xi|^{2}|\Re z-\Re w)|^{2}\geq |\xi|^{2}/2$, thus we  may disregard the part of $\Re Q(z,w,\zeta,\epsilon)$ depending on $\epsilon$.
Thanks to \eqref{qe4bnr0} we have that
  \begin{align}\label{aliru76765}
    -\Re \langle \zeta \rangle  \Re f'(t^{\ast})  \leq |\xi|B \delta \quad\textrm{ and }\quad
    \Im \langle \zeta \rangle  \Im f'(t^{\ast})\leq \frac{1}{2} |\xi| B \delta,
  \end{align}
for every $ w\in \overline{U}_\delta,  z\in B_r^\C(x_0)$.

% We will also need the following trivial inequalities
%   \begin{align}
%     - (\Im z- \Im w)\xi \leq |\Im z - \Im w||\xi| \leq (r+ \delta) |\xi|, \quad \forall w\in \overline{U}_\delta, \, z\in B_r^\C(x_0).
%   \end{align}
%  
Moreover, it is clear that either $|\Re z- \Re w|> \lambda$ or $|\Re w- \tau|> \lambda$, otherwise, $|\Re z- \tau|< 2\lambda$ which is a contradiction with our choices of $\lambda$ and $r_0$, see \eqref{aksueyt}.
It follows from \eqref{ex:2} and \eqref{laueyrt}  that
\begin{align}
  \Re \langle \zeta\rangle p(\tau - \Re w) &\geq c' \frac{\sqrt{3}}{2}|\xi| |\tau- \Re w|^{2k}
  \nonumber\\
  &\geq \left\{\begin{array}{ll}
            c' \lambda^{2k}\frac{\sqrt{3}}{2} |\xi|, \, &\textrm{ when } |\tau -\Re w|> \lambda\\
            0, \, &\textrm{ otherwise.}
  \end{array}
  \right.\label{g4ir098}
  \end{align}
  Also, we have
  \begin{align}
     s|\xi||\Re z-\Re w|^2 \geq \left\{\begin{array}{ll}
            s |\xi|\lambda^{2}, \, &\textrm{ when } |\Re z-\Re w|>\lambda\\
            0, \, &\textrm{ otherwise.}
  \end{array}\right. \label{g4isas8}
  \end{align}

Grouping \eqref{qe4bnr0}, \eqref{semponsro}, \eqref{aliru76765},   \eqref{g4ir098} and \eqref{g4isas8} we obtain, when $|\tau -\Re w|> \lambda$, that
\begin{align}\label{aieurh6as}
\Re Q(z,w,\zeta,\epsilon)&\le- \Big(c'\lambda^{2k}\frac{\sqrt{3}}{2}  -r- \delta- \frac{3}{2} B\delta \Big)|\xi|,
\end{align}
%
%Moreover, it follows from \eqref{qe4bnr0}, \eqref{semponsro}, \eqref{aliru76765}, \eqref{alir236765},   \eqref{g4ir098} and \eqref{g4isas8} that, 
and in the case where
 $|\Re z- \Re w|> \lambda$, it holds
\begin{align}\label{aieurh6}
\Re Q(z,w,\zeta,\epsilon)&\le-  \Big(s\lambda^{2} -r- \delta - \frac{3}{2} B\delta \Big)|\xi|.
\end{align}

Finally, it is clear from the expressions \eqref{aieurh6as} and \eqref{aieurh6} that one can further reduce $r$ and $\delta$, if necessary, to guarantee that there exists a positive constant $c$ independent of $\epsilon>0$ so that 
\begin{equation}\label{anjenfo}
\Re Q(z,w,\zeta,\epsilon)\le -c|\xi|, 
\end{equation}
for all $w\in \overline{U}_\delta, \, z\in B_r^\C(x_0)$.
In conclusion, \eqref{anjenfo} guarantees the
 integrability of $\omega_\epsilon$ over $(\del \Omega)_3^{\ast}$, independently of $\epsilon>0$ and therefore,  $|\mu_4^{\epsilon}|$ is bounded uniformly in $\epsilon>0$. This completes the proof.
\end{proof}

\section{A FBI classification of others types of wave-front set}\label{AFBIClassification}

Given a subsheaf $\mathcal{S}$ of $\mathcal{B}$ assume that we have already defined a condition $\mathfrak{M}(\mathcal{S})$ of {\it microlocal regularity in $V\times \Gamma$ with respect to $\mathcal{S}$.} Established this, we can define conditions $(\mathfrak{R}_1^{\mathcal{S}})$ and $(\mathfrak{R}_2^{\mathcal{S}})$ at a covector $(x_0, \xi_0) \in \RN\times (\RN\setminus \{0\})$ for a hyperfunction $u \in \mathcal{B}(\RN)$ in the following way:
\begin{itemize}
\item [$(\mathfrak{R}_1^{\mathcal{S}})$] There exists an open neighborhood $U$  of $x_0$ in $\RN$ such that, for every open neighborhood $V\subset\subset U$ of $x_0$ and every  representative $\mu \in \mathcal{O}'(\overline{U})$ of $u$ in $U$, there are a conic neighborhood $\Gamma$ of $\xi_0$ in $\RN$ such that $\mu$ satisfies condition $\mathfrak{M}(\mathcal{S})$ in $V\times \Gamma.$
\item [$(\mathfrak{R}_2^{\mathcal{S}})$]
  There exist  open neighborhoods $U, V$  of $x_0$ in $\RN$, $V\subset\subset U,$ a  representative $\mu \in \mathcal{O}'(\overline{U})$ of $u$ in $U$, a conic neighborhood $\Gamma$ of $\xi_0$ in $\RN$ such that $\mu$ satisfies condition $\mathfrak{M}(\mathcal{S})$ in $V\times \Gamma.$
\end{itemize}

Of course, $(\mathfrak{R}_1^{\mathcal{S}})$ implies $(\mathfrak{R}_2^{\mathcal{S}})$ for every sheaf $\mathcal{S}$.

\subsection{Microlocal Smoothness}

Let us now consider $\mathcal{S}= \Cinf$. In this context, we say that a hyperfunction $u$ is microlocally smooth or microlocally $\Cinf$-regular and denote its  smooth wave-front set  by  $WF_{\Cinf}(u)$. 
Let $\mu\in \mathcal{O}'(\overline{U})$. Given an open subset  $V\subset\subset U$  and a cone $\Gamma\in \RN$  we will say that $\mu$ satisfies the condition $\mathfrak{M}(\Cinf)$ of {\it microlocal smooth regularity in $V\times \Gamma$} if for every integer $q$ there is $C_q>0$ such that
  \begin{align} \label{DesCinf}
  \tag{$\mathfrak{M}(\Cinf)$}
    |\mathcal{F}_{p}\mu(\tau, \xi)|\leq C_q (1+|\xi|)^{-q}, \quad \forall (\tau, \xi) \in V\times \Gamma.
  \end{align}

\begin{Teo}\label{SmoothTheo}  Let $u \in \mathcal{B}(\RN)$, $x_0 \in \RN$ and $\xi_0 \in \RN\setminus\{0\}$. Then $(x_0, \xi_0)\notin WF_{\Cinf}(u)$ if and only if condition $(\mathfrak{R}_1^{\Cinf})$ holds at $(x_0, \xi_0)$ if and only if condition $(\mathfrak{R}_2^{\Cinf})$ holds at $(x_0, \xi_0)$.
\end{Teo}  
\begin{proof}
  Let us assume that $(x_0, \xi_0) \notin WF_{\Cinf}(u)$ and prove that $(\mathfrak{R}_1^{\Cinf})$ holds in $(x_0, \xi_0)$. 
 Using Definition~\ref{alsiu980}, we can find $\ell$ open convex cones $\Gamma_1, \ldots, \Gamma_\ell\subset \RN\setminus \{0\}$ and functions $F_1, \ldots, F_\ell$, with $F_j \in \mathcal{O}(\mathcal{W}_\delta(U; \Gamma_j)),$ satisfying properties $(1)$ and $(2)$ of microlocal smoothness. Define $v= u - \sum_{j=1}^{\ell}\bb F_j\in \Cinf(U)$ and let $\mu$ and $\nu$ be representatives in $U$ of $u$ and $v,$ respectively. Note that $\nu$ is an analytic functional defined by a smooth function in $U$, let us denote it by $f$. Thanks to Proposition~\ref{FBIBHBV}, there exist $V$ and $\Gamma$ such that \eqref{DesCinf} holds for $\mu$ if and only if it holds for $\nu$. 
  %Now $\nu$ is an analytic functional defined by a smooth function in $U$, let's call it $f$.
  
  Choose $W$ an open subset of $\RN$ satisfying $V\subset\subset W\subset\subset U$ along with $\varphi\in \Cinf_c(W)$ such that $\varphi=1$ in $V.$ We have that $\nu- \nu_{\varphi f} \in \mathcal{O}'(\overline{U}\setminus V).$ Thus
  \begin{align}\label{daifuertrtr}
    \mathcal{F}_p \nu(\tau, \xi)= \mathcal{F}_p(\nu-\nu_{\varphi f})(\tau, \xi)+ \mathcal{F}_p \nu_{\varphi f}(\tau, \xi).
  \end{align}
  Since $\nu-\nu_{\varphi f}$ is an analytic functional carried by $\overline{U}\setminus V$  Proposition~\ref{Invariance} implies that its FBI transform satisfies \eqref{DesCinf} in $V\times \RN.$ Finally, note that
  \begin{align}\label{bhwekuyg}
    \mathcal{F}_p \nu_{\varphi f}(\tau, \xi)&= c_p \int_{U} \varphi(w) f(w) e^{i(\tau-w)\xi- |\xi|p(\tau-w)} \dd w
    \notag\\
    &=\mathcal{F}_p \big(\varphi f\big)(\tau, \xi),
  \end{align}
where the  FBI transform appearing in last line of \eqref{bhwekuyg} is in the sense of distributions. It now follows from \cite[Theorem 3.1]{BH} that for every positive integer $q$ there is a constant $C_q>0$ such that
  \eqref{DesCinf} is valid.

Now let us assume condition $(\mathfrak{R}_2^{\Cinf})$ at $(x_0, \xi_0)$ and prove that $(x_0, \xi_0) \notin WF_{\Cinf}(u)$.  Thanks to Proposition~\ref{FundamentalRemark}, it will be enough to show that the boundary value $\bb F_1$ is a smooth function in a neighborhood of $x_0$. Note that $F_1$ is a holomorphic function in $\RN+ i \Gamma_1^{c}$ that has a smooth extension to $\RN$:
  \begin{align*}
      F_1(x) =\frac{1}{(2\pi)^{N}} \int_{\mathcal{C}_1} \int_{|\tau-x_0|\leq a} e^{i\xi(x-\tau)} \mathcal{F}_{p}\mu (\tau, \xi)|\xi|^{\frac{N}{2k}} \dd \tau \dd \xi.
  \end{align*}
  Let $\nu_{F_1, W'}\in \mathcal{O}'(\overline{W'})$ be defined by
  \begin{align*}
   \nu_{F_1,W'}(h)= \int_{W'} F_1(x) h(x) \dd x, \quad \forall h \in \mathcal{O}(\CN).
  \end{align*}

 It remains to show that  $\mu_{F_1, W'}- \nu_{F_1, W'}\in \mathcal{O}'(\del W')$. To do so choose an open neighborhood $\mathcal{U}$ of $\del W'$ and $y_1 \in \Gamma_1^{c}$ such that $\mu_{F_1, W'}- \mu _{F_1,W'}^{y_1}\in \mathcal{O}'(\mathcal{U}).$ Let us prove that $\mu _{F_1,W'}^{y_1}-\nu_{F_1, W'}\in \mathcal{O}'(\mathcal{U}).$ To this end, let $h \in \mathcal{O}(\CN)$ and denote by $\sigma$ the measure induced on $\del W$ and compute
  \begin{align*}
    \big(\mu _{F_1,W'}^{y_1}-\nu_{F_1,W'}\big)(h)&= \int_{W'} \big(F_1(x+iy_1) h(x+iy_1)- F_1(x) h(x)\big) \dd x\\
    &= \int_{0}^{1} \int_{\del W'} F_1(x+it y_1) h(x+i t y_1) \dd \sigma(x) \dd t,
  \end{align*}
 which is an analytic functional carried by $\mathcal{U}$, as we wished to prove.
\end{proof}

\subsection{Microlocal regularity in Denjoy-Carleman quasianalytic and  non-quasianalytic classes of Roumieu type}

In this subsection we  will consider the regularity according sheaves associate to Denjoy-Carleman quasianalytic  and non-quasianalytic spaces
which are given in  Definition~\ref{regDenjoyCarlemanfucntionsdef}, recall that we are always assuming that $M=(M_k)_{k\in \Z_+}$ is a regular sequence as in Definition \ref{regseqdefinition}.
In this context, we say that a hyperfunction $u$ is microlocally  $\mathcal E^{\{M\}}$-regular  and denote its wave-front set  by $WF_{\mathcal{E}^{\{M\}}}(u)$.

The condition $\mathfrak{M}(\mathcal{E}^{\{M\}})$  of { \it microlocal regularity in $V\times \Gamma$  with respect to the classes $\mathcal{E}^{\{M\}}$} for an analytic functional $\mu$  is defined using the associated function $M(t)$ (see \eqref{asssssssociate}) in the following way: there are constants $C,c>0$  such that
  \begin{align}\label{DesCM}
  \tag{$\mathfrak{M}(\mathcal{E}^{\{M\}})$}
    |\mathcal{F}_{p}\mu(\tau, \xi)|\leq C e^{-M(c|\xi|)}, \quad \forall (\tau, \xi) \in V\times \Gamma.
  \end{align} 

 \begin{Teo}\label{DC-QA-NQA}
Let $u \in \mathcal{B}(\RN)$, $x_0 \in \RN$ and $\xi_0 \in \RN\setminus\{0\}$. Then $(x_0, \xi_0)\notin WF_{\mathcal{E}^{\{M\}}}(u)$ 
 if and only if condition $(\mathfrak{R}_1^{\mathcal{E}^{\{M\}}})$ holds at $(x_0, \xi_0)$ if and only if condition $(\mathfrak{R}_2^{\mathcal{E}^{\{M\}}})$ holds at $(x_0, \xi_0)$.
\end{Teo}
  
%Due to technical reasons %(compare \eqref{DesCM} with \eqref{DesCMA}) 
%we will divide the proof into two parts.

%\subsubsection{Proof of Theorem~\ref{DC-QA-NQA}, in the non-quasianalytic case.}
\begin{proof}
Let us assume that $(x_0, \xi_0)\notin WF_{\mathcal{E}^{\{M\}}}(u)$. Repeating the procedure in the proof of Theorem~\ref{SmoothTheo} and using the same notation we can write, analogously as in \eqref{daifuertrtr},
  \begin{align}\label{daifuer675t}
    \mathcal{F}_p \nu(\tau, \xi)= \mathcal{F}_p(\nu-\nu_{\varphi f})(\tau, \xi)+ \mathcal{F}_p \nu_{\varphi f}(\tau, \xi).
  \end{align}
where $W$ and $V$ are open subsets of $\RN$ satisfying $V\subset\subset W\subset\subset U$, $\varphi\in \Cinf_c(W)$ is such that $\varphi=1$ in $V$ and $\nu$ is an analytic functional defined by a  function $ f\in \mathcal{E}^{\{M\}}(U)$.
  Since $\nu-\nu_{\varphi f}$ is an analytic functional carried by $\overline{U}\setminus V$,  Proposition~\ref{Invariance} implies that its FBI transform satisfies \eqref{DesCM} in $V\times \RN.$ 
  Also, taking advantage of \eqref{bhwekuyg}, $\mathcal{F}_p \nu_{\varphi f}$ can be identified as the FBI transform in the sense of distributions. The proof now follows from \cite[Theorem 5.1]{Furdos}.

Let us now assume that the microlocal condition $(\mathfrak{R}_2^{\mathcal{E}^{\{M\}}})$ is valid at $(x_0, \xi_0)$ and we shall prove that $(x_0, \xi_0) \notin WF_{\mathcal{E}^{\{M\}}}(u)$. Repeating the argument of the smooth case, all we need to do is to prove is that
  \begin{align*}
    F_1(x)= \frac{1}{(2\pi)^{N}}\int_{\mathcal{C}_1} \int_{|\tau-x_0|\leq a} e^{i\xi(x-\tau)} \mathcal{F}_{p}\mu (\tau, \xi)|\xi|^{\frac{N}{2k}} \dd \tau \dd \xi
  \end{align*}
  is a well-defined function in $\mathcal{E}^{\{M\}}(\RN)$. 

Let $\ell$ be a positive integer such that $|\xi|^{\frac{N}{2k}- \ell}$ is integrable in $\RN$. Thus, using that
\begin{align*}
  \frac{t^r}{M_r} = e^{\log \frac{t^r}{M_r}}\leq e^{ M(t)}
\end{align*}
together with \eqref{M2'inducao} it follows that 
  \begin{align*}
     \int_{\mathcal{C}_1}  e^{-M(c|\xi|)}|\xi|^{\frac{N}{2k}+|\alpha|} \dd \xi&= c^{-|\alpha|-\ell} M_{|\alpha|+ \ell} \int_{\mathcal{C}_1}  e^{-M(c|\xi|)}|\xi|^{\frac{N}{2k}-\ell} \frac{ (c|\xi|)^{\ell+|\alpha|}}{M_{|\alpha|+\ell}} \dd \xi\\
&\leq M_{|\alpha|+ \ell} c^{-|\alpha|-\ell} \int_{\mathcal{C}_1} |\xi|^{\frac{N}{2k}-\ell}  \dd \xi\\
&\leq \tilde{C} \bigg( \frac{H^{\ell}}{c}\bigg)^{|\alpha|} M_{|\alpha|},
  \end{align*}
where 
\begin{align*}
  \tilde{C}=  A^{\ell} H^{ \ell(\ell+1)/2} c^{-\ell} \int_{\mathcal{C}_1} |\xi|^{\frac{N}{2k}-\ell}  \dd \xi.
\end{align*}

Therefore,
  \begin{align*}
    \sup_{x \in \RN} |\del_x^{\alpha} F_1(x)|  &\leq  \frac{1}{(2\pi)^{N}}\int_{\mathcal{C}_1} \int_{|\tau-x_0|\leq a} \big| \mathcal{F}_{p}\mu (\tau, \xi)\big||\xi|^{\frac{N}{2k}+|\alpha|} \dd \tau \dd \xi\\
    &\leq \frac{C}{(2\pi)^N} m(B_{a}(x_0)) \int_{\mathcal{C}_1}  e^{-M(c|\xi|)}|\xi|^{\frac{N}{2k}+|\alpha|} \dd \xi\\
                                           &\leq \frac{C \tilde{C}}{(2\pi)^N} m(B_{a}(x_0)) \bigg( \frac{H^{\ell}}{c}\bigg)^{|\alpha|} M_{|\alpha|}.
  \end{align*}
 This concludes the proof of the theorem.
\end{proof}

\subsection{Microlocal regularity for distributions}

Since $\mathcal{D}'$ is embed in $\mathcal{B}$ there is not obstruction to analyze the microlocal regularity of a hyperfunction with respect to $\mathcal D'$ and, again, we can measure this regularity using the generalized FBI transform. The condition $\mathfrak{M}(\mathcal{D}')$ of { \it microlocal regularity in $V\times \Gamma$ with respect to $\mathcal{D}'$} for an analytic functional $\mu$ is defined in the following way: there are an integer $q$ and a constant $C>0$ such that
  \begin{align}\label{DesDist}
  \tag{$\mathfrak{M}(\mathcal{D}')$}
    |\mathcal{F}_{p}\mu(\tau, \xi)|\leq C(1+|\xi|)^{q}, \quad \forall (\tau, \xi) \in V\times \Gamma.
  \end{align} 
  
 \begin{Teo}  Let $u \in \mathcal{B}(\RN)$, $x_0 \in \RN$ and $\xi_0 \in \RN\setminus \{0\}$. Then $(x_0, \xi_0) \notin WF_{\mathcal{D}'}(u)$ if and only if  condition $(\mathfrak{R}_1^{\mathcal{D}'})$ holds at $(x_0, \xi_0)$ if and only if condition $(\mathfrak{R}_2^{\mathcal{D}'})$ holds at $(x_0, \xi_0).$
\end{Teo}  

\begin{proof}
  Assume that $(x_0, \xi_0) \notin WF_{\mathcal{D}'}(u)$ and let us prove the validity of $(\mathfrak{R}_1^{\mathcal{D}'})$. According with Definition~\ref{alsiu980} we can find $\ell$ open convex cones $\Gamma_1, \ldots, \Gamma_\ell\subset \RN\setminus \{0\}$ opposing the direction  $\xi_0$ and functions $F_1, \ldots, F_\ell$, with $F_j \in \mathcal{O}(\mathcal{W}_\delta(U; \Gamma_j))$ and such that $u - \sum_{j=1}^{\ell}\bb F_j$ is a hyperfunction defined by a distribution, which we will denote it by $v$.

  Let $\nu_v$ be the analytic functional associated with $v$ in $U$. According with Corollary~\ref{WFSSBVH}, it will be enough to prove estimate \eqref{DesDist} for the FBI transform of $\nu_v$. Now choose $W$ an open subset of $\RN$ satisfying $V\subset\subset W\subset\subset U$ along with $\varphi\in \Cinf_c(W)$ satisfying $\varphi=1$ in $V.$ We have that $\nu_v- \nu_{\varphi v} \in \mathcal{O}'(\overline{U}\setminus V)$ and
  \begin{align*}
    \mathcal{F}_p \nu_v(\tau, \xi)= \mathcal{F}_p(\nu_v-\nu_{\varphi v})(\tau, \xi)+ \mathcal{F}_p \nu_{\varphi v}(\tau, \xi).
  \end{align*}
Since $\nu_v-\nu_{\varphi v}$ is an analytic functional carried by $\overline{U}\setminus V$ and thanks to Proposition \ref{Invariance} we see that its FBI transform has exponential decay in $V\times\RN$. Since $\varphi v$ has compact support, it has finite order which we will denote it by $q$. Thus we have the following estimate:
  \begin{align}
    |\mathcal{F}_p \nu_{\varphi v}(\tau, \xi)|&= |c_p| \big|v_x\big(\varphi(x) e^{i(\tau-x)\xi- |\xi|p(\tau-x)} \big)\big| \nonumber\\
                                              &\leq |c_p| \sum_{\substack{\alpha \in \ZN_+\\ |\alpha|\leq q}} \sup_{x\in U} \big|\del^{\alpha} \big( \varphi(x) e^{i(\tau-x)\xi- |\xi|p(\tau-x)} \big)\big| \nonumber\\
                                              &\leq |c_p| \sum_{\substack{\alpha \in \ZN_+\\ |\alpha|\leq q}} \sup_{x \in U} \bigg| \sum_{\beta\leq \alpha} {\alpha \choose \beta} \big[\del^{\beta}_x \varphi(x)\big] \del^{\alpha-\beta}_x\big(e^{i(\tau-x)\xi- |\xi|p(\tau-x)} \big)\bigg|\nonumber\\
    &\leq C (1+ |\xi|)^{q},
  \end{align}
as we wished to prove.
  
Next, assume that $(\mathfrak{R}_2^{\mathcal{D}'})$ holds at $(x_0, \xi_0)$ and let us prove that $(x_0, \xi_0)\notin WF_{\mathcal{D}'}(u)$. Again, we will avail ourselves of Proposition~\ref{FundamentalRemark}. Thus, our goal now is to define a distribution that agrees with $\bb F_1$ in a neighborhood of $x_0$. Let $W'$ be a bounded open set in $\RN$ with smooth boundary such that $U\subset\subset W'$.  We claim that
\begin{align*}
     \Cinf_c(\RN)  \ni\phi \mapsto\frac{1}{(2\pi)^{N}} \int_{\RN}\bigg(\int_{\mathcal{C}_1} \int_{|\tau-x_0|\leq a} e^{i\xi(x-\tau)} \mathcal{F}_{p}\mu (\tau, \xi)|\xi|^{\frac{N}{2k}} \dd \tau \dd \xi\bigg) \phi(x) \dd x
  \end{align*}
defines a distribution $v_{F_1}$ in $\RN$. In fact, fix $\phi\in \Cinf_c(\RN)$ and note that
\begin{align*}
  v_{F_1}(\phi) &=\frac{1}{(2\pi)^{N}}\int_{\RN}\bigg(\int_{\mathcal{C}_1} \int_{|\tau-x_0|\leq a} e^{i\xi(x-\tau)} \mathcal{F}_{p}\mu (\tau, \xi)|\xi|^{\frac{N}{2k}} \dd \tau \dd \xi\bigg) \phi(x) \dd x\\
  &=\frac{1}{(2\pi)^{N}}\int_{\mathcal{C}_1} \int_{|\tau-x_0|\leq a} e^{-i\xi\tau} \widehat{\phi}(-\xi)\mathcal{F}_{p}\mu (\tau, \xi)|\xi|^{\frac{N}{2k}} \dd \tau \dd \xi, 
\end{align*}
where $\widehat{\phi}$ stands for the Fourier transform of $\phi$. So $v_{F_1}$ is a linear functional and one can easily check its continuity.

Now choose $\varphi \in \Cinf_c(W')$, $\varphi=1$ in $U$ e define $\nu_{\varphi v_{F_1}}\in \mathcal{O}'(\overline{W'})$ as
\begin{align}
\nu_{\varphi v_{F_1}}(h)= v_{F_1}(\varphi h), \quad \forall h \in \mathcal{O}(\CN).
\end{align}

Let us prove that $\mu_{F_1,W'}$ and $\nu_{\varphi F_1}$ define the same analytic functional in $U$, i.e., given $\epsilon>0$ we need to prove that $\mu_{F_1,W'}- \nu_{\varphi F_1}\in \mathcal{O}'((W'\setminus U)_{\epsilon})$. If we choose $y_1 \in \Gamma_1^{c}$ with $|y_1|< \epsilon$, then $\mu_{F_1, W'}- \mu_{F_1, W'}^{y_1}\in \mathcal{O}'((W'\setminus U)_\epsilon)$. We can write $\mu_{F_{1}, W'}^{y_1}- \nu_{\varphi F_1}= \lambda_1+\lambda_2-\nu_{\varphi F_1}$ where
\begin{align*}
  \lambda_1(h)=  \frac{1}{(2\pi)^{N}}\int_{W'}\int_{\mathcal{C}_1} \int_{|\tau- x_0|\leq a} e^{i\xi(x+iy_1-\tau)} \mathcal{F}_{p} \mu(\tau, \xi) |\xi|^{\frac{N}{2k}} \varphi(x) h(x+iy_1)\dd \tau \dd \xi \dd x
\end{align*}
and
\begin{align*}
  \lambda_2(h)=  \frac{1}{(2\pi)^{N}}\int_{W'\setminus U}\int_{\mathcal{C}_1} \int_{|\tau- x_0|\leq a} e^{i\xi(x+iy_1-\tau)} \mathcal{F}_{p} \mu(\tau, \xi) |\xi|^{\frac{N}{2k}} (1-\varphi(x)) h(x+iy_1)\dd \tau \dd \xi \dd x.
\end{align*}
Note that $\lambda_2$ is carried by $(W'\setminus U)_\epsilon$. Let
\begin{align*}
\mathcal{W}= \{z\in \CN: \Re z \in W', \exists t \in [0,1]: \Im z= t y_1\},
\end{align*}
and consider the smooth $N$-form $\omega(z)= e^{i\xi(z-\tau)}\varphi(\Re z) h(z)\dd z$. Now we apply Stokes' theorem to obtain
\begin{align*}
  (\lambda_1 -\nu_{\varphi F_1})(h)&= \frac{1}{(2 \pi)^{N}}\sum_{j=1}^{N} \int_{\mathcal{C}_1} \int_{|\tau- x_0|\leq a} \int_{\mathcal{W}} e^{i\xi(z-\tau)} \mathcal{F}_{p} \mu(\tau, \xi) |\xi|^{\frac{N}{2k}} \frac{\del\varphi}{\del \z_j}(\Re z) h(z) \dd \z_j\wedge \dd z \dd \tau \dd \xi .
\end{align*}
Since $\varphi= 1$ in $U$, $\lambda_1- \nu_{\varphi F_1}$ is carried by $(W'\setminus U)_\epsilon$.
\end{proof}

\subsection{Microlocal regularity for non-quasianalytic ultradistributions of Roumieu type}

Let $\mathcal{D}'^{\{M\}}$  denote the sheaf of ultradistributions, i.e., given $U$ an open subset of $\RN,$ $\mathcal{D}'^{\{M\}}(U)$  is the topological dual of $\mathcal{D}^{\{M\}}(U)$ as in Definition~\ref{Defultradis}, we recall that $M= (M_k)_{k\in \Z_+}$ is a regular sequence.

The embedding of the sheaf of ultradistributions into the sheaf of hyperfunctions can be done in the same fashion as we did for the sheaf of distributions and the main difference is now $\varphi$ has to be chosen in $\mathcal{D}^{\{M\}}(U)$ (see \cite{Ko} for details).

In this context, we say that a hyperfunction $u$ is microlocally  $\mathcal{D}'^{\{M\}}$-regular and denote its  wave-front set  by  $WF_{\mathcal{D}'^{\{M\}}}(u).$ We say that an analytic functional $\mu$ satisfies the condition $\mathfrak{M}(\mathcal{D}'^{\{M\}})$ { \it of microlocal regularity in $V\times \Gamma$ with respect to $\mathcal{D}'^{\{M\}}$} if, for every $L>0,$ there is  $C=C_L>0$  such that
  \begin{align}\label{DesDualCM}
  \tag{$\mathfrak{M}(\mathcal{D}'^{\{M\}})$}
    |\mathcal{F}_{p}\mu(\tau, \xi)|\leq C e^{M(L|\xi|)}, \quad \forall (\tau, \xi) \in V\times \Gamma
  \end{align} 
where $M(t)$ is the associated function given by \eqref{asssssssociate}.

 \begin{Teo} 
 Let $u \in \mathcal{B}(\RN)$, $x_0 \in \RN$ and $\xi_0 \in \RN\setminus \{0\}$. Then $(x_0, \xi_0)\notin WF_{\mathcal{D}'^{\{M\}}}(u)$  if and only if  condition $(\mathfrak{R}_1^{\mathcal{D}'^{\{M\}}})$  holds at $(x_0, \xi_0)$ if and only if condition $(\mathfrak{R}_2^{\mathcal{D}'^{\{M\}}})$   holds at $(x_0, \xi_0).$
\end{Teo}  
\begin{proof}
  Assume that $(x_0, \xi_0) \notin WF_{\mathcal{D}'^{\{M\}}}(u)$ and let us verify the validity of $(\mathfrak{R}_1^{\mathcal{D}'^{\{M\}}})$ at $(x_0, \xi_0)$. The procedure described in the distributions case also applies here and therefore it is enough to prove 
that estimate \eqref{DesDualCM}  holds for $\mathcal{F}_p \nu_{ v}(\tau, \xi)$ where $v\in \mathcal{E}'^{\{M\}}(U)$, i.e., $v$ is an ultradistribution of class $\{M\}$ with compact support in $U$. Using \eqref{ultradistributioncompactDefinition}, we have
  \begin{align*}
    |\mathcal{F}_p \nu_{ v}(\tau, \xi)|&= |c_p| \big|v_x\big( e^{i(\tau-x)\xi- |\xi|p(\tau-x)} \big)\big|\\
                                              &\leq |c_p| \sum_{\alpha \in \ZN_+} C_{\epsilon} \frac{\epsilon^{|\alpha|}}{M_\alpha} \sup_{x\in U} \big|\del^{\alpha}_x \big( e^{i(\tau-x)\xi- |\xi|p(\tau-x)} \big)\big|.
  \end{align*}

 Let $\kappa(\xi, \tau, x)= i(\tau-x)\xi- |\xi|p(\tau-x)$. Note that $\del_x^{\alpha} \kappa=0$ if $|\alpha|>2k,$ since $\kappa$ is homogeneous of order $1$ in $\xi$ there is a constant $C>0$ such that all derivatives of $\kappa$ are uniformly bounded by $|\xi|C$, when $x\in\overline U$ and $\tau\in\overline V$. If $|\xi|\geq 1$, it follows from the Fa\`a di Bruno formula (see e.g. \cite{bierstonemilman}) that
  \begin{align*}
    |\del^{\alpha}_x e^{\kappa(\xi,\tau, x)}|&= \Big|\sum_{j=1}^{|\alpha|} \frac{e^{\kappa}}{j!} \sum_{\substack{\gamma_1+ \cdots+ \gamma_j= \alpha \\|\gamma_k|\geq 1}} { \alpha \choose \gamma_1, \ldots, \gamma_j} (\del^{\gamma_1}_x \kappa)\ldots (\del^{\gamma_j}_x \kappa)\Big|\\
                                          &\leq \sum_{j=1}^{|\alpha|} \frac{1}{j!} \sum_{\substack{\gamma_1+ \cdots+ \gamma_j= \alpha\\ |\gamma_k|\geq 1}} 2^{|\alpha|} |\del^{\gamma_1}_x \kappa|\ldots |\del^{\gamma_j}_x \kappa|\\
                                            &\leq |\xi|^{|\alpha|} e^{C} 2^{|\alpha|(N+2)}.                                           
  \end{align*}

Therefore, given $L>0$ choose $\epsilon>0$ such that $\epsilon 2^{N+2}/L < 1/2$ then it follows that
\begin{align*}
    |\mathcal{F}_p \nu_{ v}(\tau, \xi)|
&\leq  |c_p| e^{C} C_\epsilon  \sum_{\alpha \in \ZN_+}\bigg(\frac{\epsilon2^{(N+2)}}{L}\bigg)^{|\alpha|} \frac{(L|\xi|)^{|\alpha|}}{M_{|\alpha|}}  \\
 &\leq  |c_p| e^{C} C_\epsilon   e^{M(L|\xi|)} 
  \end{align*}
as we wished to prove.

For the other implication, assume that  $(\mathfrak{R}_2^{\mathcal{D}'^{\{M\}}})$ is valid at $(x_0, \xi_0)$ and we shall prove that $(x_0, \xi_0) \notin WF_{\mathcal{D}'^{\{M\}}}(u)$. Now we will prove that $\bb F_1$ agrees with an ultradistribution in a neighborhood of $x_0$ and so the theorem will be consequence of Proposition~\ref{FundamentalRemark}.

Define the ultradistribution  $v_{F_1}\in \mathcal{D}'^{\{M\}}(\RN)$  by
  \begin{align*}
     \mathcal{D}^{\{M\}}(\RN)  \ni\phi \mapsto\frac{1}{(2\pi)^{N}} \int_{\RN}\bigg(\int_{\mathcal{C}_1} \int_{|\tau-x_0|\leq a} e^{i\xi(x-\tau)-\epsilon|\xi|^{2}} \mathcal{F}_{p}\mu (\tau, \xi)|\xi|^{\frac{N}{2k}} \dd \tau \dd \xi\bigg) \phi(x) \dd x.
  \end{align*}
 In fact,  choose $\ell\in \Z_+$ such that $|\xi|^{\frac{N}{2k}- 2\ell}$ is integrable in $\RN$ and 
fix $\phi\in \mathcal{D}^{\{M\}}(\RN)$. Therefore, it follows that
\begin{align*}
  v_{F_1}(\phi) &=\frac{1}{(2\pi)^{N}}\int_{\RN}\bigg(\int_{\mathcal{C}_1} \int_{|\tau-x_0|\leq a} e^{i\xi(x-\tau)} \mathcal{F}_{p}\mu (\tau, \xi)|\xi|^{\frac{N}{2k}} \dd \tau \dd \xi\bigg) \phi(x) \dd x\\
  &=\frac{1}{(2\pi)^{N}}\int_{\mathcal{C}_1} \int_{|\tau-x_0|\leq a} e^{-i\xi\tau} \widehat{\phi}(-\xi)\mathcal{F}_{p}\mu (\tau, \xi)|\xi|^{\frac{N}{2k}} \dd \tau \dd \xi\\
&=\frac{1}{(2\pi)^{N}}\int_{\mathcal{C}_1} \int_{|\tau-x_0|\leq a} e^{-i\xi\tau} \widehat{\Delta_{x}^\ell\phi}(-\xi)\mathcal{F}_{p}\mu (\tau, \xi)|\xi|^{\frac{N}{2k}-2 \ell} \dd \tau \dd \xi. 
\end{align*}
Since $\phi$ has compact support there is $r>0$ such that $\|\phi\|_{r, \overline{U}}<\infty$  and thus we can use \eqref{DesLapDCEst}, \eqref{FourierContinuityinDC} and \eqref{DesDualCM} with $L= (\sqrt{N}r)^{-1}$  to obtain
\begin{align*}
  |v_{F_1}(\phi)|&\leq \frac{m(B_a(x_0))}{(2\pi)^{N}}\int_{\mathcal{C}_1}  | \widehat{\Delta_{x}^\ell\phi}(-\xi)||\mathcal{F}_{p}\mu (\tau, \xi)||\xi|^{\frac{N}{2k}-2 \ell}  \dd \xi\\
&\leq \frac{C}{(2\pi)^{N}}m(B_a(x_0)) m(U) \|\Delta^\ell \phi\|_{H^{2\ell}r, \overline{U}}\int_{\mathcal{C}_1}    |\xi|^{\frac{N}{2k}-2 \ell}  \dd \xi\\
&\leq  \tilde{C}\|\phi\|_{r, \overline{U}},
\end{align*}
where
\begin{align*}
  \tilde{C}= \frac{C}{(2\pi)^{N}}m(B_a(x_0)) m(U) r^{2\ell}H^{\ell(2\ell+1)} A^{2 \ell}N\int_{\mathcal{C}_1}    |\xi|^{\frac{N}{2k}-2 \ell}  \dd \xi.
\end{align*}

To finish the proof one can repeat the same procedure as in the distribution case. 
\end{proof}

\section{Elliptic regularity}\label{secEllipticReg}

Let $\Omega\subset \RN$ be an open subset and let $P$ be a linear differential operator of order $m$ and real-analytic coefficients in $\Omega$. 

Let us denote by $P_m$ the principal symbol of $P$ and consider the characteristic set of $P$, i.e,
\begin{align}\label{charP}
  \Char P= \{(x, \xi) \in T^\ast \Omega: P_m(x,\xi)= 0\}.
\end{align}

The main result of this section  is  the elliptic regularity theorem for some subsheaves of the sheaf of hyperfunctions. 

\begin{Teo}\label{Cor73}
  Let $\mathcal{S}$ one of the following sheaves: $C^{\omega}, \Cinf, \mathcal{E}^{\{M\}}, \mathcal D'$. For every $u \in \mathcal{B}(\Omega)$ and every linear differential operator $P$ of order $m$ with real-analytic coefficients in $\Omega$ it follows that
  \begin{align*}
    WF_{\mathcal{S}}(u) \subset WF_{\mathcal{S}}(Pu) \cap \Char P.
  \end{align*}
\end{Teo}

In the next lines we will  use the microlocal decomposition, Definition~\ref{MLdecom}, to proof  the case $\mathcal S=\mathcal D'$, Theorem~\ref{TeoRegElip}. Then we conclude the proof in Subsection~\ref{conlcusionP} as a consequence of Theorem~\ref{TeoRegElip}.

Fix $u \in \mathcal{B}(\Omega)$ and assume that $(x_0, \xi_0)\notin WF_{\mathcal{S}}(Pu) \cup \Char P$ where $P$ is a linear differential operator of order $m$ and real-analytic coefficients and $\mathcal{S}$ is a subsheaf of $\mathcal{B}$ that has $C^{\omega}$ as a subsheaf.
Our goal is to prove that $(x_0, \xi_0)\notin WF_{\mathcal{S}}(u).$ 

Let $U$ be a bounded open neighborhood of $x_0$ in $\Omega$ and assume that $\mu, \nu\in \mathcal{O}'(\overline{U})$ are representatives of $u$ and $Pu$  in $U$.

Using the microlocal decomposition of $\mu$ and $\nu$ we can find cones $\Gamma_1, \ldots, \Gamma_\ell$ such that $\xi_0 \in \Gamma_1$, $\xi_0 \cdot\Gamma_j<0$ if $j=2, \ldots, \ell$; and functions $F_j \in \mathcal{O}(\mathcal{W}_\delta(U; \Gamma_j))$, $H_j \in \mathcal{O}(\mathcal{W}_\delta(U;\Gamma_j))$, $1\le j\le \ell$, such that
  \begin{align*}
    u= \sum_{j=1}^{\ell} \bb F_j \quad \textrm{ and } \quad  P(x,D)u= \sum_{j=1}^{\ell} \bb H_j.
  \end{align*}

We also know that
\begin{align*}
  F_1(z)= \frac{1}{(2\pi)^{N}} \int_{|\tau-x_0|\leq  a} \int_{\mathcal{C}_1} e^{i\xi(z-\tau)} \mathcal{F}_p \mu (\tau, \xi) |\xi|^{\frac{N}{2k}} \dd \xi \dd \tau
\end{align*}
and
\begin{align*}
    H_1(z)= \frac{1}{(2\pi)^{N}} \int_{|\tau-x_0|\leq  a} \int_{\mathcal{C}_1} e^{i\xi(z-\tau)} \mathcal{F}_p \nu (\tau, \xi) |\xi|^{\frac{N}{2k}} \dd \xi \dd \tau
\end{align*}
 Note that $(x_0, \xi_0) \notin WF_{\mathcal{S}}(u)$ if and only if $(x_0, \xi_0) \notin WF_{\mathcal{S}}(\bb F_1)$.

The next result is a particular case of Theorem~\ref{Cor73} when $\mathcal S=\mathcal D'$.

\begin{Teo} \label{TeoRegElip}
  For every $u \in \mathcal{B}(\Omega)$ and every linear differential operator $P$ of order $m$ with real-analytic coefficients in $\Omega$ it follows that
  \begin{align*}
    WF_{\mathcal{\mathcal{D}'}}(u) \subset WF_{\mathcal{D}'}(P u) \cup \Char P.
  \end{align*}
\end{Teo}
\begin{proof}
  From the microlocal decomposition, Definition~\ref{MLdecom}, it follows that 
  \begin{align*}
    P \bb F_1- \bb H_1 = \sum_{j=2}^{\ell} (\bb H_j- P \bb F_j).
  \end{align*}
  Since $(x_0, \xi_0) \notin \Char P$ we can, reducing $U$ and $\Gamma_1$, assume  that $P$ is elliptic on $U+i\Gamma_1.$ 

In Appendix \ref{Appendix3} we prove that there are holomorphic functions $E_{H_1}$ and $S_{H_1}$ such that $P \bb E_{H_1}= \bb H_1+ \bb S_{H_1}$ where $\bb E_{H_1} \in \mathcal{D}'(U)$ and $\bb S_{H_1}\in C^{\omega}(U)$.

Therefore we obtain that
\begin{align*}
  P(\bb F_1-  \bb  E_{H_1})= \sum_{j=2}^{\ell} (\bb H_j - P \bb F_j)- \bb S_{H_1}
\end{align*}

Thanks to \cite[Theorem 9.3.7]{Hormander1}, we have that there are $g_2, \ldots, g_\ell\in \mathcal{B}(U)$ such that $WF_a(g_j) \subset U \times (\Gamma_1\cap \Gamma_j)$ and
\begin{align*}
  P (\bb F_1- \bb E_{H_1}) = \sum_{j=2}^{\ell} g_j.
\end{align*}

Note that $\Gamma_1$ can be choose such that $\Gamma_1\cap \Gamma_j = \emptyset$ for every $j=2, \ldots, \ell$. Thus, we can assume that $g_2, \ldots, g_\ell$ are real-analytic.  Now,  \cite[Corollary 9.4.9]{Hormander1} allows us to find, for every $j=2, \ldots, \ell$,  holomorphic functions  $U_j $ defined in complex neighborhood of $x_0$   such that  $P(z, D) \bb U_j= g_j$ in a neighborhood of $x_0$. From \cite[Theorem 9.3.3]{Hormander1}, we know that 
\begin{align*}
  \bb \Big(P(z, D_z) \big( F_1-  E_{H_1} - \sum_{j=2}^{\ell} U_j\big)\Big)=0
\end{align*}
implies
\begin{align*}
  P(x, D) \big( F_1-  E_{H_1} - \sum_{j=2}^{\ell} U_j\big)=0.
\end{align*}
Now by the second part of \cite[Corollary 9.4.9]{Hormander1} we conclude that $ F_1-  E_{H_1} - \sum_{j=2}^{\ell} U_j$ has an extension as a holomorphic function in a neighborhood of the origin. Therefore, there is $W$ a neighborhood of $x_0$ such that
\begin{align*}
  \bb F_1- \sum_{j=2}^{\ell} \bb U_j- \bb \big( F_1-  E_{H_1} - \sum_{j=2}^{\ell} U_j\big)= \bb E_{H_1} \in \mathcal{D}'(W)
\end{align*}
Thus $(x_0, \xi_0)\notin WF_{\mathcal{D}'}(\bb F_1)$ and, consequently,  $(x_0, \xi_0) \notin WF_{\mathcal{D}'}(u)$.
\end{proof}

\subsection{Proof of Theorem~\ref{Cor73}}\label{conlcusionP}
Assume that $\mathcal{S}$ is a subsheaf of $\mathcal{D}'$. We
know that $WF_{\mathcal{D}'}(u)\subset WF_{\mathcal{D}'}(Pu) \cup \Char P$, Theorem~\ref{TeoRegElip}. Thus, if $(x_0, \xi_0)\notin WF_{\mathcal{S}}(Pu) \cup \Char P$ then $(x_0, \xi_0)\notin WF_{\mathcal{D}'}(u)$. Thus we can assume that $\bb F_1 \in \mathcal{D}'(U)$. Since the elliptic regularity theorem holds for distributions when we consider their regularity with respect to $\mathcal{S}$, \cite{Hormander1}, we have
\begin{align}\label{know}
  WF_{\mathcal{S}}(\bb F_1) \subset WF_{\mathcal{S}}(P \bb F_1) \cup \Char P.
\end{align}
Additionally, if $j= 2, \ldots, \ell,$ then $(x_0, \xi_0)\notin  WF_{\mathcal{S}}(P \bb F_j)=  WF_{\mathcal{S}}\big( \bb (P F_j)\big)$, because $P(z, D_z)F_j(z)$ is a holomorphic function defined in a cone opposing the direction $\xi_0$. Therefore,  $(x_0, \xi_0)\notin WF_{\mathcal{S}}(Pu) \cup \Char P$ implies that $(x_0, \xi_0) \notin WF_{\mathcal{S}}(P \bb F_1) \cup \Char P$ which together with \eqref{know} give us that $(x_0, \xi_0) \notin  WF_{\mathcal{S}}(\bb F_1)$. Consequently, by the microlocal decomposition, we obtain that $(x_0, \xi_0) \notin   WF_{\mathcal{S}}(u)$. \qed

%Let us summarize the  preceding arguments in connection with Theorem~\ref{TeoRegElip} a posteriori in the following observation:
%\begin{Obs}\label{Obs72}
%If $\mathcal{S}$ is a subsheaf of $\mathcal{D}'$, then  it is sufficient to prove the elliptic regularity theorem in the distribution setting to conclude its validity for hyperfunctions. 
%\end{Obs}
\medskip

The techniques in the proof of Theorem~\ref{Cor73}, allow us to state a more general result.

\begin{Cor}\label{Cor73???}
Let $P$ be a linear differential operator of order $m$ with real-analytic coefficients in an open set $\Omega\subset\R^N$.
  Let $\mathcal{S}$ be a subsheaf of $\mathcal D'$ in which the elliptic regularity theorem holds, that is, 
  \begin{align*}
  WF_{\mathcal{S}}(u) \subset WF_{\mathcal{S}}(P u) \cup \Char P,\quad \forall u\in\mathcal D'(\Omega)
\end{align*}
   then for every $u \in \mathcal{B}(\Omega)$  it follows that
  \begin{align*}
    WF_{\mathcal{S}}(u) \subset WF_{\mathcal{S}}(Pu) \cup \Char P.
  \end{align*}
\end{Cor}

\begin{Exe}
Consider $\omega$ a weight function and $\mathcal E_{\{\omega\}}(\Omega)$ the $\omega$-ultradifferentiable functions of Roumieu
type as defined in \cite{BMT}. Since the elliptic regularity theorem holds for distributions \cite{AJO}, then for every $u \in \mathcal{B}(\Omega)$ it hold
  \begin{align*}
    WF_{\mathcal E_{\{\omega\}}}(u) \subset WF_{\mathcal E_{\{\omega\}}}(Pu) \cup \Char P.
  \end{align*}

\end{Exe}
\section{Final comments}

%Quasianalytic Wave Front Sets for Solutions of Linear Partial Differential Operators A. A. Albanese, D. Jornet and A. Oliaro
% R.W. Braun, R. Meise, B.A. Taylor, Ultradifferentiable functions and Fourier analysis, Results Math. 17 (1990) 206–237.

\subsection{Microlocal regularity in every direction}

  One may ask if $(x_0, \xi) \notin WF_{C^{\omega}}(u)$ for every $\xi\in \RN$, then $u$ is real-analytic in a neighborhood of $x_0$.  Using the compactness of the unitary sphere in $\RN$ we can use this condition to prove that there is $V\subset\RN$ such that
  \begin{align*}
    |\mathcal{F}_p\mu(\tau,\xi)|\leq c_1 e^{-c_2|\xi|}, \quad \forall(\tau, \xi) \in V\times \RN.
  \end{align*}

  Therefore, when we write $\mu_{\epsilon}(x)= \mu_0^{\epsilon}(x)+ \mu_1^{\epsilon}(x)$, as in \eqref{oieur} and \eqref{oieurgbi}, then Proposition~\ref{ConvergenceLemma} gives that $\mu_1$ is real-analytic in a neighborhood of $x_0$, whereas $\mu_0^{\epsilon}$ converges to a real-analytic function by the same arguments given in the proof of Proposition~\ref{laiw}.
Analogously, the same is true for $\Cinf, \mathcal{E}^{\{M\}},\mathcal{D}', {\mathcal{D}'}^{\{M\}}$.
So, at least for any of the sheaves treated in this work, if $u \in \mathcal{B}(U)$ and $WF_{\mathcal{S}}(u)= \emptyset$, then
 that $u\in \mathcal{S}(U).$

\subsection{Recovering the definition of Wave-front set in the sense of distributions}

Assume that $u\in \mathcal{D}'(U)$ and $\mathcal{S}= \Cinf, C^{\omega}, \mathcal{E}^{\{M\}}$. The definition of wave-front set with respect to $\mathcal{S}$ that we adopted in this work is different from the classical one. They are, of course, equivalent and the key to see that is that both are completely characterized by the behavior of the FBI transform. Since $u\in \mathcal{D}'(U)$ the behavior of the FBI transform is independent if we are identifying the distribution with an analytic functional or not, we conclude that the definitions are equivalent.

\subsection{Microlocal regularity in Denjoy-Carleman classes of Beurling type}
Similarly to \eqref{DesCM} and \eqref{DesDualCM} one can define the microlocal conditions for the spaces of Denjoy-Carleman classes of ultradifferentiable functions (quasi and non-quasianalytic) and Denjoy-Carleman  non-quasianalytic classes of ultradistributions of Beurling type.
Given an analytic functional $\mu \in \mathcal{O}'(\overline{U})$, $V\subset \subset U$ and a cone $\Gamma\in \RN$, we say that $\mu$ satisfies the condition
$\mathfrak{M}(\mathcal{E}^{(M)})$  of{ \it microlocal regularity in $V\times \Gamma$  with respect to the classes} $\mathcal{E}^{(M)}$ if for every positive constant $\rho$ there exists  $C_\rho>0$  such that
  \begin{align}\label{DesCM2}
  \tag{$\mathfrak{M}(\mathcal{E}^{(M)})$}
    |\mathcal{F}_{p}\mu(\tau, \xi)|\leq C_\rho e^{-M(\rho|\xi|)}, \quad \forall (\tau, \xi) \in V\times \Gamma.
  \end{align} 

And we say that $\mu$ satisfies the condition $\mathfrak{M}(\mathcal{D}'^{(M)})$ of{ \it microlocal regularity in $V\times \Gamma$ with respect to $\mathcal{D}'^{(M)}$} if there exist constants $C,c>0$  such that
  \begin{align}\label{DesDualCM2}
  \tag{$\mathfrak{M}(\mathcal{D}'^{(M)})$}
    |\mathcal{F}_{p}\mu(\tau, \xi)|\leq C e^{M(c|\xi|)}, \quad \forall (\tau, \xi) \in V\times \Gamma.
  \end{align}

Here $M(t)$ is the associated function given by \eqref{asssssssociate}. The proof that microlocal regularity regarding the sheaf of Denjoy-Carleman functions of Beurling type is equivalent to both conditions $(\mathfrak{R}_1^{\mathcal{E}^{(M)}})$ and $(\mathfrak{R}_2^{\mathcal{E}^{(M)}})$ follows the same ideas as before with one additional step needed: to prove the local regularity characterization using FBI transform estimates for Denjoy-Carleman  functions of Beurling type for distributions and believe it can be done following the ideas in \cite{HM, Furdos}. The analogous result for Denjoy-Carleman ultradistributions of Beurling type can be done using the computations in the proof of the Roumieu case, see Remark~\ref{Obs:Beurling}.

\appendix
\section{}

\subsection{ A remark about cones.}\label{AppA1}
Let $y \in \R^{N}\setminus\{0\}$ and consider $C_r^{y}= \{ x \in S^{N-1}: d\big(x, \frac{y}{|y|}\big)< r\}$ for simplicity we assume that $r$ is small in order to have the following property: $x\in C_r^y$ then $-x \notin C_r^y$. We consider $\Gamma_r^y$ the cone generated by $C_r^y$, i.e., $\Gamma_r^y= \{ x \in \RN: x/|x| \in C_r^y\}.$ 
Choose $\xi_0\in \RN\setminus\{0\}$. There is $s>0$ such that $y- s \xi_0 \in \Gamma_r^y$ and so there is $t>0$ such that $y- \xi$ belongs to a compact set $K$ of $\Gamma_r^y$ for every $\xi\in \Gamma^{\xi_0}_t$ with $|\xi|\leq s.$ Now define $\lambda= s/ |y|$. Thus, we have that $y- \xi\in K$ for every $\xi\in \Gamma^{\xi_0}_t$ with $|\xi|\leq \lambda|y|.$ Now note that if $y'\in \RN$ is such that $\rho y'= y$, where $\rho>0,$ then $y' - \xi\in \rho^{-1} K$ for every $\xi\in \Gamma^{\xi_0}_t$ with $|\xi|\leq \lambda |y'|$. This last comment is important because it proves that we can reduce the norm of $y$ and the property still holds for the same cones and same constant $\lambda$ (and for a smaller compact set of $\Gamma^{y}_r$). Observe that if $\xi_0 \cdot y< 0$ we can choose $t$ such that $\xi \cdot y<0$ for every $\xi\in \Gamma^{\xi_0}_t$.

\subsection{Denjoy-Carleman space of ultradifferentiable functions}\label{SA3}

In this work we considered the spaces of Denjoy-Carleman functions defined by sequences. 
Let $M=(M_k)_{k\in\Z_+}$ be a sequence of positive real numbers. We will impose some condition on $(M_k)_{k\in \Z_+}$ in order to invoke the results from  \cite{Furdos} and \cite{Ko}.

There is the standard initial condition,
 \begin{equation} \label{P1} 
M_0=M_1=1.
\end{equation}
We will also assume that the sequence $(M_k/k!)_{k\in Z_+}$ is logarithmic convex, i.e.,
\begin{equation} \label{logconvex1}
\bigg(\frac{M_k}{k!}\bigg)^2 \leq \frac{M_{k-1}}{(k-1)!} \frac{ M_{k+1}}{(k+1)!} , \quad \forall k \geq 1. 
\end{equation}
This implies that $(M_k)_{k\in Z_+}$ is logarithmic convex, i.e.,
\begin{equation} \label{logconvex}
M_k^2 \leq M_{k-1}  M_{k+1} , \quad \forall k \geq 1. 
\end{equation}
Our third hypothesis on the sequence $(M_k)_{k\in\Z_+}$ is the stability under differential operators condition, there are $A, H>0$ such that:
\begin{equation}\label{M2'} 
M_{k}\leq A H^{k} M_{k-1} \quad \forall k\in \Z_+.
\end{equation} 
The last condition is
\begin{equation}
  \label{M4}
   \sqrt[\leftroot{-1}\uproot{2}\scriptstyle k]{\frac{M_k}{k!}} \lra  \infty.
\end{equation}

\begin{Def}[Regular sequences]\label{regseqdefinition} A sequence $M= (M_k)_{k\in \Z_+}$ satisfying \eqref{P1}, \eqref{logconvex1}, \eqref{M2'} and \eqref{M4} will be called a regular sequence.
\end{Def}

\begin{Def}[Denjoy-Carleman functions]\label{regDenjoyCarlemanfucntionsdef}
Consider a sequence of positive numbers $M=(M_k)_{k\in \Z+}$.  Let $\Omega$ be an open subset of $\RN$ and consider $f\in \Cinf(\Omega)$. We shall say that $f$ is a Denjoy-Carleman (ultradifferentiable) function of class  $M$ if for every $K\subset \Omega$ compact there are $C>0$ and $r>0$ such that
\begin{equation}
  \sup_{x\in K}|\del_x^\alpha f(x)|\leq C r^{|\alpha|}M_{|\alpha|},
\end{equation}
for every $\alpha\in \ZN_+$, this is equivalent to say that
\begin{align}
  \|f\|_{r, K}:= \sup_{\alpha\in \Z_+}\sup_{x\in K}\frac{ |\del_x^\alpha f(x)|}{r^{|\alpha|}M_{|\alpha|}}< \infty.
\end{align}
  We will denote the space of Denjoy-Carleman functions of class $M$ in $\Omega$ by $\mathcal{E}^{\{M\}}(\Omega)$.
{\it In this work we will only consider Denjoy-Carleman functions defined by regular sequences.}
\end{Def}

\begin{Rem}\label{Obs:Beurling}
The space in Definition~\ref{regDenjoyCarlemanfucntionsdef} is often called Denjoy-Carleman space of Roumieu type. The space of smooth functions, $f\in \Cinf(\Omega)$, such that for every $K\subset \Omega$ compact and every $r>0$ $ \|f\|_{r, K}< \infty$  is called the Denjoy-Carleman space of Beurling type. For simplicity, all the results bellow are stated and proved to the Denjoy-Carleman spaces of Roumieu type but they can be extended to the Denjoy-Carleman spaces of Beurling type and it will be left as an exercise to the reader.
\end{Rem}

 Let us discuss some consequence of these hypotheses. Property \eqref{M4} implies that $C^{\omega}(U)$ is a proper subspace of $\mathcal{E}^{\{M\}}(U)$.
If the sequence $M=(M_k)_{k\in\Z_+}$ satisfies \eqref{P1} and \eqref{logconvex},
then $\mathcal E^{\{M\}}$ is closed  under compositions, see \cite{bierstonemilman}.

The stability  under differential operators yields the following inequality for the terms of the sequence:
\begin{align}\label{M2'inducao}
  M_{k+\ell} & \leq (A H^{k})^{\ell} H^{\ell(\ell+1)/2} M_{k},
\end{align}
where $A$ and $H$ are given by \eqref{M2'}.

Denjoy-Carleman theorem says that $\mathcal{E}^{\{M\}}(\Omega)$ is quasianalytic if and only if
\begin{equation}
  \sum_{k=1}^{\infty} \frac{1}{M_k^{1/k}} = \infty.
\end{equation}

Note that the Gevrey sequences $M_p=p!^s$ satisfies  \eqref{P1},  \eqref{logconvex1}, \eqref{M2'}  and \eqref{M4} for all $s> 1$, moreover, they are non-quasianalytic. The Gevrey space associated with $p!$ is the space of real-analytic functions and, in this case, \eqref{P1},  \eqref{logconvex1} and \eqref{M2'} hold and the class  is quasianalytic. 

 When $M$ is non-quasianalytic class we denote by $\mathcal{D}^{\{M\}}(U)$ the subspace of the compactly supported functions in $\mathcal{E}^{\{M\}}(U).$ If $\phi\in \mathcal{D}^{\{M\}}(U)$, there is $r>0$  such that
\begin{align}\label{DesDCcompactosupp}
  \|\phi\|_{r, \overline{U}}:= \sup_{\alpha\in \Z_+}\sup_{x\in \overline{U}}\frac{ |\del_x^\alpha \phi(x)|}{r^{|\alpha|}M_{|\alpha|}}< \infty.
\end{align}
Sometimes is important to know the $r>0$ for which \eqref{DesDCcompactosupp} holds. Let us fix  $\phi\in \mathcal{D}^{\{M\}}(U)$ satisfies \eqref{DesDCcompactosupp} for $r>0$, then we can use \eqref{M2'inducao} to obtain
\begin{align}
  \|\Delta^\ell \phi\|_{H^{2\ell}r, \overline{U}}&= \sup_{\alpha\in \ZN_+} \sup_{\overline{U}}\frac{|\del^\alpha \Delta^{\ell}\phi|}{(H^{2\ell} r)^{|\alpha|} M_{|\alpha|}}\nonumber\\
&\leq \sup_{\alpha\in \ZN_+} \sup_{\overline{U}}\frac{|\del^\alpha (\Delta^{\ell}\phi)|}{ r^{|\alpha|+2\ell} M_{|\alpha|+2 \ell}} \frac{ r^{2\ell} M_{|\alpha|+2 \ell}}{ H^{2\ell|\alpha|} M_{|\alpha|}}\nonumber\\
&\leq r^{2\ell}H^{\ell(2\ell+1)} A^{2 \ell}N \|\phi\|_{r, \overline{U}}< \infty. \label{DesLapDCEst}
\end{align}

If $M=(M_k)_{k\in \Z_+}$ is a sequence of positive numbers we define the associated function $M: (0,\infty) \lra  (0,\infty)$ by
\begin{align}\label{asssssssociate}
M(t) := \sup_{k\in \Z_+}\log \frac{t^k }{M_k}.
\end{align}
Let us now use the associate function to prove the continuity of the Fourier transform acting on a function $\phi\in \mathcal{D}^{\{M\}}(U)$. First, note that
\begin{align}
  |\xi^{\alpha} \widehat{\phi}(\xi)|&= \bigg| \int_{U} \del^\alpha_x e^{i x \xi} \phi(x) \dd x\bigg|\nonumber\\ 
&\leq m(U) \|\phi\|_{r, \overline{U}}r^{|\alpha|}M_{|\alpha|}.\label{Desigualdaporbaixo}
\end{align}

If $\xi\in \RN\setminus\{0\}$, then there is $j\in \{1, \ldots, N\}$ such that $\xi_j^2\geq \xi_k^2$ for every $k\in \{1, \ldots, N\}.$ So we can define $\tilde{\alpha}= |\alpha| e_j$, note that $|\xi^{\tilde{\alpha}}|= |\xi_j|^{|\alpha|}\geq  (|\xi|/\sqrt{N})^{|\alpha|}$, thus we can use \eqref{Desigualdaporbaixo} to obtain
\begin{align*}
  |\widehat{\phi}(\xi)|&\leq m(U) \|\phi\|_{r, \overline{U}}\frac{(\sqrt{N}r)^{|\alpha|}M_{|\alpha|} }{|\xi|^{|\alpha|}}\\
&\leq m(U) \|\phi\|_{r, \overline{U}}\exp\bigg[ - \log \bigg(\frac{|\xi|^{|\alpha|}}{\big(\sqrt{N}r\big)^{|\alpha|}M_{|\alpha|} }\bigg)\bigg].
\end{align*}
Thus
\begin{align}
  |\widehat{\phi}(\xi)| &\leq m(U) \|\phi\|_{r, \overline{U}}\exp \inf_{k\in \Z_+}\Big[ - \log \Big(\frac{|\xi|^{k}}{(\sqrt{N}r)^{k}M_{k} }\Big)\Big]\nonumber\\
&\leq m(U) \|\phi\|_{r, \overline{U}}e^{- M\big(\frac{|\xi|}{\sqrt{N}r}\big)}. \label{FourierContinuityinDC}
\end{align}

When $\mathcal{E}^{\{M\}}(U)$ is non-quasianalytic and defined by a regular sequence $M= (M_k)_{k\in \Z_+}$ we can proceed as in the distribution theory and consider the dual space of $\mathcal{D}^{\{M\}}(U)$.
\begin{Def}[Denjoy-Carleman ultradistributions] \label{Defultradis}
  We will say that a linear functional $u: \mathcal{D}^{\{M\}}(U)\lra \C$ is an ultradistribution of class $M$ if, for every $r>0$ and every $K\subset U$ compact there is $C_r>0$ such that
  \begin{align}\label{ultradistributionDefinition}
    |u(\phi)|\leq C_r \|\phi\|_{r, K},
  \end{align}
for every $\phi\in \mathcal{D}^{\{M\}}(U)$ (we are working under the convention  that if $\|\phi\|_{r, K}= \infty$ then \eqref{ultradistributionDefinition} is automatically satisfied). The space of all ultradistributions in $U$ of class $M$ is denoted by $ {\mathcal{D}^{\{M\}}}'(U)$.
One can identified the space of the ultradistributions of class $M$ with compact with the dual space of $\mathcal{E}^{\{M\}}(U)$, i.e., the space of all linear functionals  $u: \mathcal{E}^{\{M\}}(U)\lra \C$ such that there is a compact $K\subset U$ such that for every $r>0$ there is $C_r>0$ such that
\begin{align}\label{ultradistributioncompactDefinition}
    |u(\phi)|\leq C_r \|\phi\|_{r, K},
\end{align}
for every $\phi\in \mathcal{E}^{\{M\}}(U)$.
{\it In this work we only consider ultradistributions that are in the dual of a non-quasianalytic class and are defined by a regular sequence.}
\end{Def}

\subsection{Existence of a microlocal parametrix} \label{Appendix3}

Consider $U$ an open subset of $\RN$, $x_0 \in U$ and $\nu \in \mathcal{O}'(\overline{U})$. Moreover assume that we have $P$ a differential operator of order $m$ with real-analytic coefficients that is elliptic in a cone $\mathcal{C}_1$. We are going to impose some growth condition on $\mathcal{F}_p\nu (\tau, \xi)$ for every $\xi \in \mathcal{C}_1$ and every $\tau$ near $x_0$, so there is $a>0$ such that the condition holds if $|\tau-x_0|\leq a$ and $\xi\in \mathcal{C}_1$.

Now we define
\begin{align*}
  H_1(z) =\int_{|\tau-x_0|\leq a} \int_{\mathcal{C}_1} e^{i\xi(z-\tau)} \mathcal{F}_p\nu (\tau, \xi) |\xi|^{\frac{N}{2k}} \dd \xi \dd \tau.
\end{align*}
Observe that $H_1$ is a holomorphic function  $\mathcal{W}(U;\Gamma_1)$ for some cone $\Gamma_1$ such that $y\cdot\xi< 0$ for all $y \in \Gamma_1$ and $\xi\in \mathcal{C}_1$.

We will  extend the Definition 1.7 from \cite{Boutet} allowing analytic symbol be defined only on $U_{\delta}\times \mathcal{C}_1$. So if $q$ is an analytic symbol defined on $U_\delta\times \mathcal{C}_1$, then we can define $G$ in $\mathcal{W}_\delta(U;  \Gamma_1)$ by
\begin{align}\label{DefinitionOfG}
  G(z)=  \int_{|\tau-x_0|\leq a} \int_{\mathcal{C}_1} q(z, \xi) e^{i\xi(z-\tau)} \mathcal{F}_p\nu (\tau, \xi) |\xi|^{\frac{N}{2k}} \dd \xi \dd \tau.
\end{align}

Observe that
\begin{align*}
  P(z, D_z) G(z)&=  \int_{|\tau-x_0|\leq a} \int_{\mathcal{C}_1} \sum_{|\alpha|\leq m} \sum_{\beta\leq \alpha} {\alpha \choose \beta} a_\alpha(z)\xi^{\alpha-\beta}  D_z^\beta q(z, \xi) e^{i\xi(z-\tau)} \mathcal{F}_p\nu (\tau, \xi) |\xi|^{\frac{N}{2k}} \dd \xi \dd \tau\\
&=  \int_{|\tau-x_0|\leq a} \int_{\mathcal{C}_1} \sum_{|\alpha|\leq m} \sum_{\beta\leq \alpha}\frac{1}{\beta!} a_\alpha(z) \del_\xi^{\beta}(\xi^{\alpha})  D_z^\beta q(z, \xi) e^{i\xi(z-\tau)} \mathcal{F}_p\nu (\tau, \xi) |\xi|^{\frac{N}{2k}} \dd \xi \dd \tau.
\end{align*}
Recall that given two symbols of pseudodifferential operators $ p(z, \xi)$ and $ q(z, \xi)$ then the symbol of the composition,  $ p(z, \xi)\circ q(z,\xi)$, is given by
\begin{align*}
  p(z, \xi)\circ q(z,\xi)= \sum_{\gamma \in \ZN_+} \frac{i^{-|\gamma|}}{\gamma!} \bigg[\bigg(\frac{\del}{\del \xi}\bigg)^{\gamma} p(z, \xi)\bigg] \bigg[\bigg( \frac{\del}{\del z}\bigg)^{\gamma} q(z, \xi)\bigg].
\end{align*}

Therefore, the action of $P(z, D_z)$ on $G(z)$ is defined by the composition of the analytic symbols $p$ and $q$.

In the special case when $q= p_m^{-1}$, we have
\begin{align}
  P(z, D_z) G(z)&=  \int_{|\tau-x_0|\leq a} \int_{\mathcal{C}_1^{\delta_0}} P(z, D_z) \bigg(\frac{ e^{i\xi(z-\tau)}}{p_m(z, \xi)}\bigg) \mathcal{F}_p\nu (\tau, \xi) |\xi|^{\frac{N}{2k}} \dd \xi \dd \tau \nonumber\\
  &= H_1(z)+ \int_{|\tau-x_0|\leq a} \int_{\mathcal{C}_1^{\delta_0}} r(z, \xi) e^{i\xi(z-\tau)} \mathcal{F}_p\nu (\tau, \xi) |\xi|^{\frac{N}{2k}} \dd \xi \dd \tau\label{Eqummaisr}
\end{align}
where
\begin{align*}
  r(z, \xi)           &=  \sum_{|\alpha|\leq m} a_{\alpha}(z) \sum_{\substack{ | \alpha- \beta| \neq m \\ \beta\leq \alpha}} \frac{1}{\beta!} \del_\xi^{\beta}\xi^{\alpha} D^{\beta}_z\bigg(\frac{1}{p_m(z, \xi)}\bigg).
\end{align*}

Thus equality \eqref{Eqummaisr} can be understood in terms of symbols of (analytic) pseudodifferential operators:
\begin{align*}
  p(z, \xi) \circ (p_m(z,\xi))^{-1}= 1+ r(z,\xi).
\end{align*}
Here we need to be careful to invoke results regarding analytic pseudodifferential operators because $(p_m(z, \xi))^{-1}$ and $r(z, \xi)$ are defined only for $z $ in a complex neighborhood of $U$, $U_\delta$, and $\xi \in \mathcal{C}_1$. Our main reference is the article from Boutet de Monvel, \cite{Boutet}. The key observation here is that his arguments (at least the ones we needed) also work for symbols defined only in cones.   

Note that $(p_m(z,\xi))^{-1}$ is a symbol of order $-m$ and $r(z, \xi)$ is a symbol of order $-1$, therefore, following the ideas from \cite{Boutet},  the formal inverse for $1+ r(z, \xi)$ is given by
\begin{align*}
  \sum_{j=0}^{\infty} (-1)^{j}r^{j}(z, \xi)
\end{align*}
here $r^{j}(z, \xi)= r(z, \xi) \circ (r^{j-1}(z, \xi))$ where $\circ$ stands for the composition of symbol of pseudodifferential operators and $r^{0}(z,\xi)=1$.

We can apply \cite[Theorem 1.23]{Boutet} to find an analytic symbol $a(z, \xi)$ defined in $U_{\delta}\times \mathcal{C}_1$ such that for every $K\subset U$ compact there are $\epsilon>0$ and $c>0$ such that 
\begin{align*}
  \Big|a(z, \xi)- (p_m(z, \xi))^{-1} \circ \sum_{j=0}^{\infty} (-1)^{j} r^{j}(z, \xi)\Big|\leq c e^{-\epsilon |\xi|}
\end{align*}
for every $z \in \{ \zeta\in \CN: d(\zeta, K)< \epsilon\}$ and $\xi\in \mathcal{C}_1.$

Set $s(z, \xi)= p(z,\xi)\circ( a(z, \xi))-1$, since $p$ is a symbol of an analytic differential operator the composition is a symbol of an  analytic pseudodifferential operator (and not only a symbol of formal  analytic pseudodifferential). Since
\begin{align*}
   p(z, \xi) \circ (p_m(z, \xi))^{-1} \circ \Big(\sum_{j=0}^{\infty} (-1)^{j} r^{j}(z, \xi)\Big)=1
\end{align*}
it follows that
\begin{align*}
   |s(z, \xi)|  &= \Big|p(z,\xi)\circ \Big( a(z, \xi)-(p_m(z, \xi))^{-1} \circ \Big(\sum_{j=0}^{\infty} (-1)^{j} r^{j}(z, \xi)\Big)\Big)\Big|
\end{align*}
thus, one can check that, for every $K\subset U$ compact, there are $\tilde{\epsilon}>0$ and $\tilde{c}>0$ such that
\begin{align}\label{Seanalitico}
  |s(z, \xi)| \leq \tilde{c} e^{-\tilde{\epsilon} |\xi|} 
\end{align}
for every $z \in \{ \tilde{z}\in \CN: d(\tilde{z}, K)< \tilde{\epsilon}\}$ and $\xi\in \mathcal{C}_1.$

Now we define
  \begin{align*}
  E_{H_1}(z)=  \int_{|\tau-x_0|\leq a} \int_{\mathcal{C}_1}  a(z, \xi) e^{i\xi(z-\tau)} \mathcal{F}_p\nu (\tau, \xi) |\xi|^{\frac{N}{2k}} \dd \xi \dd \tau
  \end{align*}
  and
\begin{align*}
  S_{H_1}(z)=  \int_{|\tau-x_0|\leq a} \int_{\mathcal{C}_1} s(z, \xi) e^{i\xi(z-\tau)} \mathcal{F}_p\nu (\tau, \xi) |\xi|^{\frac{N}{2k}} \dd \xi \dd \tau.
 \end{align*}  
 
 We note that
 \begin{align*}
   P(z, D) E_{H_1}(z)= H_1(z)+ S_{H_1}(z)
 \end{align*}
 and that $S_{H_1}$ is holomorphic in a complex neighborhood of $U$. To see that, fix $x \in U$, so \eqref{Seanalitico} for $K= \{x\}$ yields $\tilde{c}>0, \tilde{\epsilon}>0$ such that $|s(z, \xi)|\leq \tilde{c}e^{- \epsilon|\xi|}$ for every $z\in B_{\tilde{\epsilon}}^{\C}(x)$, proving the $S_{H_1}$ is holomorphic in $B_{\tilde{\epsilon}/2}^{\C}(x)$. 

We conclude that
\begin{align*}
  P(x,D) \bb E_{H_1}= \bb H_1+ \bb S_{H_1}
\end{align*}
in $\mathcal{B}(U)$, where $\bb S_{H_1}\in C^{\omega}(U)$.

Now assume that there are $C>0$ and $M>0$ such that
\begin{align*}
  |\mathcal{F}_p\nu (\tau, \xi)|\leq C(1+|\xi|)^M
\end{align*}
for every $\tau$ such that $|\tau-x_0|\leq a$ and $\xi \in \mathcal{C}_1.$

Let us prove that $\bb E_{H_1}\in \mathcal{D}'(U).$ Fix $y\in \RN$, $|y|< \delta$, and, for $\varphi\in \Cinf_c(U)$, consider
\begin{align*}
  \int_{U} E_{H_1}(x+i y) \varphi(x) \dd x&= \int_{U} \int_{|\tau-x_0|\leq a} \int_{\mathcal{C}_1} a(x+iy,\xi)\varphi(x) e^{i\xi(x+iy -\tau)} \mathcal{F}_p\nu (\tau, \xi) |\xi|^{\frac{N}{2k}}  \dd \xi \dd \tau \dd x\\
&= \int_{U} \int_{|\tau-x_0|\leq a} \int_{\mathcal{C}_1} \Delta_x^{\ell}\big( a(x+iy,\xi)\varphi(x)\big) \frac{ e^{i\xi(x+iy -\tau)}}{|\xi|^{2\ell}} \mathcal{F}_p\nu (\tau, \xi) |\xi|^{\frac{N}{2k}}  \dd \xi \dd \tau \dd x
\end{align*}
so choosing $\ell$ great enough we can assure the integrability independently of $y$ and then $\bb E_{H_1}$ is a distribution of order $2 \ell.$

% \bibliographystyle{alpha}
% \bibliography{Bibliografia}

\end{document}